\documentclass{siamart0516}
\usepackage{amsfonts}
\usepackage{accents}
\usepackage[margin=1in, letterpaper]{geometry}
\usepackage{graphicx}
\usepackage{tabularx}
\usepackage{bmpsize}
\usepackage{subcaption}
\usepackage{float}
\usepackage{epsfig}
\usepackage{multicol}
\usepackage{cite}
\usepackage{adjustbox,lipsum}
\usepackage{bm}

\newcommand{\vertiii}[1]{{\left\vert\kern-0.25ex\left\vert\kern-0.25ex\left\vert #1\right\vert\kern-0.25ex\right\vert\kern-0.25ex\right\vert}}

\begin{document}

\title{Robust and Efficient Modular Grad-Div Stabilization}
\author{J. A. Fiordilino\thanks{University of Pittsburgh, Department of Mathematics, Pittsburgh, PA 15260.  The research presented herein was partially supported by NSF grants CBET 1609120 and DMS 1522267.  Moreover, J.A.F. would like to acknowledge support from the DoD SMART Scholarship.} \and W. Layton\footnotemark[1] \and Y. Rong\footnotemark[1] \thanks{Xi'an Jiaotong University, School of Mathematics and Statistics, Xi'an, Shaanxi 710049, China.  Support from NSFC grants 11171269 and 11571274 and China Scholarship Council grant 201606280154.}}
\maketitle

\begin{abstract}
	This paper presents two modular grad-div algorithms for calculating solutions to the Navier-Stokes equations (NSE).  These algorithms add to an NSE code a minimally intrusive module that implements grad-div stabilization.  The algorithms do not suffer from either breakdown (locking) or debilitating slow down for large values of grad-div parameters.  Stability and optimal-order convergence of the methods are proven.  Numerical tests confirm the theory and illustrate the benefits of these algorithms over a fully coupled grad-div stabilization.
\end{abstract}
\section{Introduction}
Grad-div stabilization of fluid flow problems has drawn attention due to its positive impact on solution quality.  However, it also introduces new computational challenges.  As the grad-div parameter $\gamma$ increases, the condition number of the resulting linear system grows without bound \cite{Glowinski}; consequently, iterative solvers can slow dramatically.  Unfortunately, appropriate values for $\gamma$ can vary wildly depending on the application; proposed values include: $\mathcal{O}(h)$ \cite{Jenkins}, $\mathcal{O}(\nu)$ \cite{Roos}, and both local and global solution ratios $\mathcal{O}(\frac{|p|_{m}}{|v|_{k+1}})$ \cite{Jenkins}, among others \cite{Decaria,John2}; values as high as $\gamma = 1,000$ and $10,000$ produce good results for Rayleigh-B\'{e}nard convection for silicon oil, \cite{Jenkins}.  Therefore, moderate or even large values of $\gamma$ may be unavoidable.  Moreover, grad-div stabilization increases coupling, decreases sparsity, and makes preconditioning more difficult.  Research has addressed the former \cite{Benzi,Boerm,Heister,LeBorne,Olshanskii,Niet,Olshanskii2} and the latter \cite{Bowers,Guermond,Linke,Schoberl}, but full resolution is still open.

This paper presents two modular grad-div stabilizations resolving both issues, Algorithms 1 and 2 in Section 3.  Algorithm 2 incorporates sparse grad-div ideas from \cite{Guermond} resulting in further storage reduction and efficiency gains.  Each method adds a minimally intrusive stabilization module.  The algorithms are simple to implement, retain the benefits of grad-div stabilization and are resilient to breakdown as stabilization parameters increase, Figure \ref{figure=speedtest} below.

\begin{figure}
	\centering
	\includegraphics[height=2.5in, keepaspectratio]{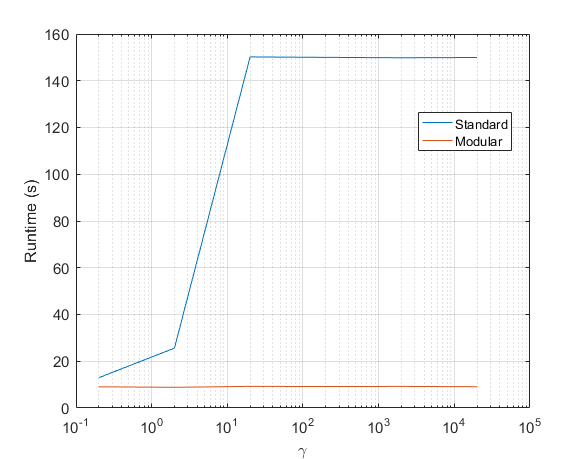}
	\caption{Runtimes: Standard vs. modular grad-div for increasing $\gamma$.}\label{figure=speedtest}
\end{figure}
Let $\Omega \subset \mathbb{R}^{d}$ (d = 2,3) be a convex polyhedral domain with piecewise smooth boundary $\partial \Omega$.  Given the fluid viscosity $\nu$, $u(x,0) = u^{0}(x)$, and the body force $f(x,t)$, the velocity $u(x,t):\Omega \times (0,t^{\ast}] \rightarrow \mathbb{R}^{d}$ and pressure $p(x,t):\Omega \times (0,t^{\ast}] \rightarrow \mathbb{R}$ satisfy
\begin{align}
u_{t} + u \cdot \nabla u - \nu \Delta u + \nabla p = f \; \; in \; \Omega, \notag
\\ \nabla \cdot u = 0 \; \; in \; \Omega, \label{s1}
\\ u = 0 \; \; on \; \partial \Omega. \notag
\end{align}

To explain how grad-div stabilization terms $-\beta \nabla \nabla \cdot u_{t}$ and $-\gamma \nabla \nabla \cdot u$ ($\beta \geq 0, \; \gamma \geq 0$) are introduced, suppress the spatial discretization momentarily and consider the following simple example:
\\ \underline{Algorithm 1}
\\ \textbf{Step 1:} Given $u^{n}$, find $\hat{u}^{n+1}$ and $p^{n+1}$ satisfying
\begin{align} \label{d1}
\frac{\hat{u}^{n+1} - u^{n}}{\Delta t} + u^{n} \cdot \nabla \hat{u}^{n+1} - \nu \Delta \hat{u}^{n+1} + \nabla p^{n+1} = f^{n+1} \; and \; \nabla \cdot \hat{u}^{n+1} = 0.
\end{align}
\textbf{Step 2:} Given $\hat{u}^{n+1}$, find $u^{n+1}$ satisfying
\begin{align} 
u^{n+1} - (\beta + \gamma \Delta t) \nabla \nabla \cdot u^{n+1} = \hat{u}^{n+1} - \beta \nabla \nabla \cdot u^{n}. \label{d2}
\end{align}
\noindent Step 1 is obviously a consistent discretization of the NSE.  Step 2 can be rewritten as
\begin{align} \label{d2rewrite}
\frac{\hat{u}^{n+1} - u^{n+1}}{\Delta t} = - \beta \nabla \nabla \cdot \big(\frac{u^{n+1} - u^{n}}{\Delta t}\big) - \gamma \nabla \nabla \cdot u^{n+1}.
\end{align}
Rewriting the first term in Step 1 as $\frac{u^{n+1} - u^{n}}{\Delta t} + \frac{\hat{u}^{n+1} - u^{n+1}}{\Delta t}$ and using (\ref{d2rewrite}), we see that Steps 1 and 2 introduce the bold terms below:
\begin{align*}
\frac{u^{n+1}-u^{n}}{\Delta t} \bm{-\beta \nabla \nabla \cdot \big(\frac{u^{n+1} - u^{n}}{\Delta t}\big)} + u^{n} \cdot \nabla \hat{u}^{n+1} - \nu \Delta \hat{u}^{n+1} \bm{- \gamma \nabla \nabla \cdot u^{n+1}} = f^{n+1}.
\end{align*}
After spatial discretization, Step 1 requires solution of a standard and well understood velocity-pressure system without the added coupling or ill conditioning of the grad-div terms while Step 2 is the same uncoupled SPD grad-div system at every timestep.  Figure \ref{figure=speedtest}, summarizing a timing test in Section 5, shows this separation into two simpler systems does result in a large efficiency increase; in this example, a speed up of $\simeq 30$ plus greater robustness (the coupled solve fails around $\gamma = 20$).

Step 2 can be combined with both sparse grad-div (from \cite{Bowers,Linke}) and lagged grad-div (from \cite{Guermond}), when $\beta \equiv 0$, for even greater efficiency.  For lagged grad-div, we replace Step 2 above:
\\ \underline{Algorithm 2}
\\ \textbf{Step 1:} Same as Step 1 in Algorithm 1.
\\ \textbf{Step 2:} Given $\hat{u}^{n+1}$, find $u^{n+1}=(u^{n+1}_{1},u^{n+1}_{2},u^{n+1}_{3})^{T}$ satisfying
\begin{align} 
u^{n+1}_{1} - \gamma \Delta t \frac{\partial^{2} u^{n+1}_{1}}{\partial x^{2}} = \hat{u}^{n+1}_{1} + \gamma \Delta t \big( \frac{\partial^{2} u^{n}_{2}}{\partial x \partial y} + \frac{\partial^{2} u^{n}_{3}}{\partial x \partial z}\big), \notag \\
u^{n+1}_{2} - \gamma \Delta t \frac{\partial^{2} u^{n+1}_{2}}{\partial y^{2}} = \hat{u}^{n+1}_{2} + \gamma \Delta t \big(\frac{\partial^{2} u^{n}_{1}}{\partial x \partial y} + \frac{\partial^{2} u^{n}_{3}}{\partial y \partial z}\big),\label{d2b} \\
u^{n+1}_{3} - \gamma \Delta t \frac{\partial^{2} u^{n+1}_{3}}{\partial z^{2}} = \hat{u}^{n+1}_{3} + \gamma \Delta t \big(\frac{\partial^{2} u^{n}_{1}}{\partial x \partial z} + \frac{\partial^{2} u^{n}_{2}}{\partial y \partial z}\big). \notag
\end{align}
Lagging the cross-terms, further uncouples the velocity components in Step 2.  Consequently, the second step, (\ref{d2b}), is solving one linear system for each component.  On a structured mesh, it can even reduce to one tridiagonal solve per meshline.

Naturally, Algorithms 1 and 2, however efficient, are only useful if they are reliable.  We prove in Section 4 that they are unconditionally, nonlinearly, energy stable.  This analysis delineates the effect of the algorithmic uncoupling on the methods consistency error, numerical dissipation and convergence properties and thereby establishes full reliability.  Our analysis is necessarily technical and thus treats the simplest (first-order) discretization (\ref{d1}) and (\ref{d2rewrite}).  Extension of both the algorithm and the analysis to some higher order methods is understudy.

In Section 2, we collect necessary mathematical tools.  In Section 3, we present fully-discrete algorithms based on (\ref{d1}) - (\ref{d2}) and a variation with (\ref{d2}) replaced by an alternative ``lagged" version of grad-div stabilization.  Stability and error analysis follow in Section 4.  In particular, unconditional, nonlinear, energy, stability of the algorithms are proven in Theorems \ref{t1} and \ref{t2} and first-order convergence in Theorems \ref{t3} and \ref{t4}.  We end with numerical experiments, which illustrate the effectiveness of these algorithms, and conclusions in Sections 5 and 6.
\subsection{Options for Step 2}
For clarity, set $\beta = 0$.  Step 2 then requires solution of the linear system arising from the variational formulation of
\begin{align}\label{illustrate}
\big(I - \gamma \Delta t \nabla \nabla \cdot \big)u = \hat{u},
\end{align}
in a velocity finite element space.  The coefficient matrix is SPD, unchanged except by a shift when the timestep is altered.  Thus, in addition to the direct methods used in our tests, good alternatives include multigrid \cite{Arnold,Hiptmair} and efficient Krylov subspace methods, e.g., \cite{Eshof,Simoncini}.  Solution via gradient flow may also suffice.  For this, a pseudotimestep $\tau$ is selected and the iteration proceeds via
\begin{align}
\frac{u_{l+1}-u_{l}}{\tau} + \frac{1}{\Delta t}u_{l+1} - \gamma \nabla \nabla \cdot u_{l} = \hat{u},
\end{align}
until satisfied.  Analysis of these (and other) options is an open problem.
\\ \noindent \textbf{Remark:}  The system (\ref{illustrate}) can be written as
\begin{align*}
\frac{u-\hat{u}}{\Delta t} - \gamma \nabla \nabla \cdot u = 0,
\end{align*}
which makes it clear that modular implementation is equivalent to operator splitting, an observation of Olshanskii and Xiong \cite{Olshanskii3} in a different context.
\section{Mathematical Preliminaries}
The $L^{2} (\Omega)$ inner product is $(\cdot , \cdot)$ and the induced norm is $\| \cdot \|$.  Define the Hilbert spaces,
\begin{align*}
X &:= H^{1}_{0}(\Omega)^{d} = \{ v \in H^{1}(\Omega)^d : v = 0 \; on \; \partial \Omega \}, \;
Q := L^{2}_{0}(\Omega) = \{ q \in L^{2}(\Omega) : (1,q) = 0 \}, \\
V &:= \{ v \in X : (q,\nabla \cdot v) = 0 \; \forall \; q \in Q \}.
\end{align*}
The explicitly skew-symmetric trilinear form is denoted:
\begin{align*}
b(u,v,w) &= \frac{1}{2} (u \cdot \nabla v, w) - \frac{1}{2} (u \cdot \nabla w, v) \; \; \; \forall u,v,w \in X.
\end{align*}
\noindent It enjoys the following properties.
\begin{lemma} \label{l1}
There exists $C_{1}$ and $C_{2}$ such that for all u,v,w $\in$ X, $b(u,v,w)$ satisfies
\begin{align*}
b(u,v,w) &= (u \cdot \nabla v, w) + \frac{1}{2} ((\nabla \cdot u)v, w), \\
b(u,v,w) &\leq C_{1} \| \nabla u \| \| \nabla v \| \| \nabla w \|, \\
b(u,v,w) &\leq C_{2}  \sqrt{\| u \| \|\nabla u\|} \| \nabla v \| \| \nabla w \|.
\end{align*}
Moreover, if $v \in H^{2}(\Omega)$, then there exists $C_{3}$ such that
\begin{align*}
b(u,v,w) &\leq C_{3} \Big( \| u \| \| v \|_{2} \| \nabla w \| + \| \nabla \cdot u \| \| \nabla v \| \| \nabla w \| \Big).
\end{align*}
Further, if $v \in H^{2}(\Omega) \cap L^{\infty}(\Omega)$, then
\begin{align*}
	b(u,v,w) &\leq \big(C_{3} \|v\|_{2} + \|v\|_{\infty}\big)\| u \| \| \nabla w \|.
\end{align*}
\begin{proof}
The identity is a calculation.  The first and second inequalities are standard.  The third inequality follows from the identity, Lemma 2.2 (g) p. 2044 in \cite{Layton2} on the convective term and the H\"{o}lder, Ladyzhenskaya, and Poincar\'{e}-Friedrichs inequalities on the second term.  The fourth follows from Lemma 2.2 (g) p. 2044 in \cite{Layton2} on the first term and H\"{o}lder's inequality with $p=q=2$ and $r=\infty$ on the second term.
\end{proof}
\end{lemma}
We define the following ``lagged" grad-div operator, which will simplify the analysis, corrresponding to \cite{Guermond}:
\begin{definition}  Let $gd_{lag}: X \times X \rightarrow \mathbb{R}$ be the lagged grad-div operator given by
	\begin{equation}\label{gd1}
	gd_{lag}(u^{n},v) = 
	\begin{cases}
	(u^{n}_{1,1},v_{1,1}) + (u^{n-1}_{2,2},v_{1,1}) + (u^{n-1}_{1,1},v_{2,2}) + (u^{n}_{2,2},v_{2,2})  & d = 2, \\
	(u^{n}_{1,1} + u^{n-1}_{2,2} + u^{n-1}_{3,3},v_{1,1}) + (u^{n-1}_{1,1} + u^{n}_{2,2} + u^{n-1}_{3,3},v_{2,2})
	\\ + (u^{n-1}_{1,1} + u^{n-1}_{2,2}+ u^{n}_{3,3},v_{3,3}) & d = 3.
	\end{cases}
	\end{equation}
\end{definition}
\indent We prove stability and convergence of lagged modular grad-div in 2d.  The analysis in 3d is an open problem.
\begin{lemma} \label{l2}
	Let $d=2$.  The following identities hold.
	\begin{align} 
	gd_{lag}(u^{n},u^{n}) = \frac{1}{2} \big\{\|u^{n}_{1,1}\|^{2} - \|u^{n-1}_{1,1}\|^{2} + \|u^{n}_{2,2}\|^{2} - \|u^{n-1}_{2,2}\|^{2}\big\} + \frac{1}{2} \big\{\|u^{n}_{1,1} + u^{n-1}_{2,2}\|^{2} + \|u^{n-1}_{1,1} + u^{n}_{2,2}\|^{2}\big\}, \notag
	\\ gd_{lag}(u^{n},u^{n})-(\nabla \cdot u^{n-1},\nabla \cdot u^{n}) = \frac{1}{2} \big\{\|u^{n}_{1,1}\|^{2} - \|u^{n-1}_{1,1}\|^{2} + \|u^{n}_{2,2}\|^{2} - \|u^{n-1}_{2,2}\|^{2}\big\} \label{gdid1}
	\\ + \frac{1}{2} \big\{\|u^{n}_{1,1} - u^{n-1}_{1,1}\|^{2} + \|u^{n}_{2,2} - u^{n-1}_{2,2}\|^{2}\big\}. \notag
	\end{align}
	\begin{proof}
		Consider (\ref{gd1}) and let $v = u^{n}$.  Then,
		\begin{align} \label{maingd}
		gd_{lag}(u^{n},u^{n}) &= \|u^{n}_{1,1}\|^{2} + \|u^{n}_{2,2}\|^{2} + (u^{n-1}_{2,2},u^{n}_{1,1}) + (u^{n-1}_{1,1},u^{n}_{2,2}).
		\end{align}
		The polarization identity applied to each of the mixed terms $(u^{n-1}_{2,2},u^{n}_{1,1})$ and $(u^{n-1}_{1,1},u^{n}_{2,2})$ yields
		\begin{align*}
		(u^{n-1}_{2,2},u^{n}_{1,1}) &= -\frac{1}{2}\big\{\|u^{n-1}_{2,2}\|^{2} + \|u^{n}_{1,1}\|^{2} - \|u^{n}_{1,1} + u^{n-1}_{2,2}\|^{2}\big\},
		\\ (u^{n-1}_{1,1},u^{n}_{2,2}) &= -\frac{1}{2}\big\{\|u^{n-1}_{1,1}\|^{2} + \|u^{n}_{2,2}\|^{2} - \|u^{n}_{2,2} + u^{n-1}_{1,1}\|^{2}\big\}.
		\end{align*}
		Using the above identities in equation (\ref{maingd}) yields the first result (\ref{gdid1}) after a rearrangement.  
		\\ \indent For the second identity, expand $(\nabla \cdot u^{n-1}, \nabla \cdot u^{n})$ and subtract it from (\ref{maingd}).  This leads to
		\begin{align*}
		gd_{lag}(u^{n},u^{n})-(\nabla \cdot u^{n-1}, \nabla \cdot u^{n}) = (u^{n}_{1,1}-u^{n-1}_{1,1},u^{n}_{1,1}) + (u^{n}_{2,2}-u^{n-1}_{2,2},u^{n}_{2,2}).
		\end{align*}
		Applying the polarization identity to each term on the r.h.s. yields the result.
	\end{proof}
\end{lemma}
The discrete time analysis will utilize the following norms $\forall \; -1 \leq k < \infty$:
\begin{align*}
\vertiii{v}_{\infty,k} &:= \max_{0\leq n \leq N} \| v^{n} \|_{k}, \;
\vertiii{v}_{p,k} := \big(\Delta t \sum^{N}_{n = 0} \| v^{n} \|^{p}_{k}\big)^{1/p}.
\end{align*}
The weak formulation of system (\ref{s1}) is:
Find $u:[0,t^{\ast}] \rightarrow X$ and $p:[0,t^{\ast}] \rightarrow Q$ for a.e. $t \in (0,t^{\ast}]$ satisfying
\begin{align}
(u_{t},v) + b(u_,u,v) + \nu (\nabla u,\nabla v) - (p, \nabla \cdot v) &= (f,v) \; \; \forall v \in X, \\
(q, \nabla \cdot u) &= 0 \; \; \forall q \in Q.
\end{align}
\subsection{Finite Element Preliminaries}
Consider a quasi-uniform mesh $\Omega_{h} = \{K\}$ of $\Omega$ with maximum triangle diameter length $h$.  Let $X_{h} \subset X$ and $Q_{h} \subset Q$ be conforming finite element spaces consisting of continuous piecewise polynomials of degrees \textit{j} and \textit{l}, respectively.  Moreover, assume they satisfy the following approximation properties $\forall 1 \leq j,l \leq k,m$:
\begin{align}
\inf_{v_{h} \in X_{h}} \Big\{ \| u - v_{h} \| + h\| \nabla (u - v_{h}) \| \Big\} &\leq Ch^{k+1} \lvert u \rvert_{k+1}, \label{a1}\\
\inf_{q_{h} \in Q_{h}}  \| p - q_{h} \| &\leq Ch^{m} \lvert p \rvert_{m}, \label{a2}
\end{align}
for all $u \in X \cap H^{k+1}(\Omega)^{d}$ and $p \in Q \cap H^{m}(\Omega)$.  Furthermore, we consider those spaces for which the discrete inf-sup condition is satisfied,
\begin{equation} \label{infsup} 
\inf_{q_{h} \in Q_{h}} \sup_{v_{h} \in X_{h}} \frac{(q_{h}, \nabla \cdot v_{h})}{\| q_{h} \| \| \nabla v_{h} \|} \geq \alpha > 0,
\end{equation}
\noindent where $\alpha$ is independent of $h$.  Examples include the MINI-element and Taylor-Hood family of elements \cite{John}.  The space of discretely divergence free functions is defined by 
\begin{align*}
V_{h} := \{v_{h} \in X_{h} : (q_{h}, \nabla \cdot v_{h}) = 0, \forall q_{h} \in Q_{h}\}.
\end{align*}
The discrete inf-sup condition implies that we may approximate functions in $V$ well by functions in $V_{h}$,
\begin{lemma} \label{l5}
	Suppose the discrete inf-sup condition (\ref{infsup}) holds, then for any $v \in V$
	\begin{equation*}
	\inf_{v_{h} \in V_{h}} \| \nabla (v - v_{h}) \| \leq C(\beta)\inf_{v_{h} \in X_{h}} \| \nabla (v - v_{h}) \|.
	\end{equation*}
\end{lemma}
\begin{proof}
	See Chapter 2, Theorem 1.1 on p. 59 of \cite{Girault}.
\end{proof}
The standard inverse inequality \cite{Ern} will be useful:
$$\| \nabla v \| \leq C_{inv} h^{-1} \| v \| \; \; \; \forall v \in X_{h},$$ 
\noindent where $C_{inv}$ depends on the minimum angle in the triangulation.
A discrete Gronwall inequality will play a  role in the upcoming analysis.
\begin{lemma} \label{l4}
(Discrete Gronwall Lemma). Let $\Delta t$, H, $a_{n}$, $b_{n}$, $c_{n}$, and $d_{n}$ be finite nonnegative numbers for n $\geq$ 0 such that for N $\geq$ 1
\begin{align*}
a_{N} + \Delta t \sum^{N}_{0}b_{n} &\leq \Delta t \sum^{N-1}_{0} d_{n}a_{n} + \Delta t \sum^{N}_{0} c_{n} + H,
\end{align*}
then for all  $\Delta t > 0$ and N $\geq$ 1
\begin{align*}
a_{N} + \Delta t \sum^{N}_{0}b_{n} &\leq exp\big(\Delta t \sum^{N-1}_{0} d_{n}\big)\big(\Delta t \sum^{N}_{0} c_{n} + H\big).
\end{align*}
\end{lemma}
\begin{proof}
	See Lemma 5.1 on p. 369 of \cite{Heywood}.
\end{proof}
\section{Numerical Schemes}
Denote the fully discrete solutions by $u^{n}_{h}$ and $p^{n}_{h}$ at time levels $t^{n} = n\Delta t$, $n = 0,1,...,N$, and $t^{\ast}=N\Delta t$.  The fully discrete approximations of (\ref{s1}) are
\\ \underline{Algorithm 1}
\\ \textbf{Step 1:}  Given $u^{n}_{h} \in X_{h}$, find $(\hat{u}^{n+1}_{h}, p^{n+1}_{h}) \in (X_{h},Q_{h})$ satisfying:
\begin{align}\label{step1}
(\frac{\hat{u}^{n+1}_{h} - u^{n}_{h}}{\Delta t},v_{h}) + b(u^{n}_{h},\hat{u}^{n+1}_{h},v_{h}) + \nu (\hat{u}^{n+1}_{h}, v_{h}) - (p^{n+1}_{h},\nabla \cdot v_{h}) = (f^{n+1},v_{h}) \; \; \forall v_{h} \in X_{h}, \\
(\nabla \cdot \hat{u}^{n+1}_{h}, q_{h}) \; \; \forall q_{h} \in Q_{h}. \notag
\end{align}
\\ \textbf{Step 2:}  Given $\hat{u}^{n+1}_{h} \in X_{h}$, find $u^{n+1}_{h} \in X_{h}$ satisfying:
\begin{align}\label{step2}
(u^{n+1}_{h},v_{h}) + (\beta + \gamma \Delta t) (\nabla \cdot u^{n+1}_{h}, \nabla \cdot v_{h}) = (\hat{u}^{n+1}_{h},v_{h}) + \beta (\nabla \cdot u^{n}_{h}, \nabla \cdot v_{h}) \; \; \forall v_{h} \in X_{h}.
\end{align}
\\ \underline{Algorithm 2}
\\ \textbf{Step 1:} Same as Step 1 in Algorithm 1.
\\ \textbf{Step 2:}  Given $\hat{u}^{n+1}_{h} \in X_{h}$, find $u^{n+1}_{h} \in X_{h}$ satisfying:
\begin{align}\label{step2b}
(u^{n+1}_{h},v_{h}) + \gamma \Delta t gd_{lag}(u^{n+1}_{h},v_{h}) = (\hat{u}^{n+1}_{h},v_{h}) \; \; \forall v_{h} \in X_{h}.
\end{align}

The linear systems resulting from either algorithm in Step 2 prescribe both the tangential and normal components of velocity on the boundary.  None-the-less, solutions exist uniquely for both algorithms and converge to the true NSE solution (Theorems \ref{t3} and \ref{t4} below).
\begin{theorem}
Consider Algorithm 1 (\ref{step1}) -( \ref{step2}) or Algorithm 2 (\ref{step1}) - ( \ref{step2b}).  Suppose $f^{n+1} \in H^{-1}(\Omega)^{d}$ and $u^{n}_{h} \in X_{h}$.  Then, there exists unique solutions $\hat{u}^{n+1}_{h}, \; u^{n+1}_{h} \in X_{h}$ to Step 1 and Step 2.
\end{theorem}
\begin{proof}
Both Step 1 and 2 reduce to solving a finite dimensional linear system after picking a basis.  Step 1 is equivalent to
\begin{align}\label{equivstep1}
(\hat{u}^{n+1}_{h},v_{h}) + \Delta t b(u^{n}_{h},\hat{u}^{n+1}_{h},v_{h}) + \nu \Delta t (\hat{u}^{n+1}_{h}, v_{h}) = \Delta t (f^{n+1},v_{h}) + (u^{n}_{h},v_{h}) \; \; \forall v_{h} \in V_{h}.
\end{align}
Existence is equivalent to uniqueness, thus we must show $\hat{u}^{n+1}_{h} \equiv 0$ provided the r.h.s. is zero.  Let $v_{h} = \hat{u}^{n+1}_{h} \in V_{h}$, then
\begin{align*}
\|\hat{u}^{n+1}_{h}\|^{2} + \nu \Delta t \|\nabla \hat{u}^{n+1}_{h}\|^{2} = 0,
\end{align*}
which implies $\hat{u}^{n+1}_{h} \equiv 0$.  Similarly, for Step 2, we must show $u^{n+1}_{h} \equiv 0$ provided the r.h.s. is zero.  Letting $v_{h} = u^{n+1}_{h} \in X_{h}$ yields
\begin{align*}
\|u^{n+1}_{h}\|^{2} + (\beta + \gamma \Delta t) \|\nabla \cdot u^{n+1}_{h}\|^{2} = 0.
\end{align*}
Uniqueness for Step 2 thus follows as well.
\end{proof}
\section{Numerical Analysis of the Modular Algorithms}
In Theorems \ref{t1} and \ref{t2}, the stability of the velocity approximations are proven for the schemes (\ref{step1}) - (\ref{step2}) and (\ref{step1}) - (\ref{step2b}).  Moreover, in Theorem \ref{t3} and \ref{t4}, first-order convergence of these algorithms is proven.

\subsection{Stability Analysis}
The following lemma is key to the stability analyses on the effect of Step 2.
\begin{lemma} \label{l3}
Consider (\ref{step2}) in Step 2 of Algorithm 1.  The following holds
\begin{multline*}
\|\hat{u}^{n+1}_{h}\|^{2} = \|u^{n+1}_{h}\|^{2} + \|\hat{u}^{n+1}_{h} - u^{n+1}_{h}\|^{2} + 2\gamma \Delta t \|\nabla \cdot u^{n+1}_{h}\|^{2} + \beta \Big(\|\nabla \cdot u^{n+1}_{h}\|^{2} - \|\nabla \cdot u^{n}_{h}\|^{2}
\\ + \| \nabla \cdot (u^{n+1}_{h} - u^{n}_{h}) \|^{2} \Big).
\end{multline*}
Moreover, consider (\ref{step2b}) in Step 2 of Algorithm 2.  Then, the following holds
\begin{multline*} 
\|\hat{u}^{n+1}_{h}\|^{2} = \|u^{n+1}_{h}\|^{2} + \|\hat{u}^{n+1}_{h} - u^{n+1}_{h}\|^{2} + \gamma \Delta t\Big(\|u^{n+1}_{1h,1}\|^{2} - \|u^{n}_{1h,1}\|^{2} + \|u^{n+1}_{2h,2}\|^{2} - \|u^{n}_{2h,2}\|^{2}\Big)
\\ + \gamma \Delta t \Big(\|u^{n+1}_{1h,1} + u^{n}_{2h,2} \|^{2} + \|u^{n+1}_{2h,2} + u^{n}_{1h,1} \|^{2}\Big).
\end{multline*}
\begin{proof}
For the first claim, let $v_{h} = u^{n+1}_{h}$ in equation (\ref{step2}) and rearrange.  Then,
\begin{equation} \label{stability:main}
(\hat{u}^{n+1}_{h},u^{n+1}_{h}) = \|u^{n+1}_{h}\|^{2} + (\beta + \gamma \Delta t) \|\nabla \cdot u^{n+1}_{h}\|^{2} - \beta (\nabla \cdot u^{n}_{h}, \nabla \cdot u^{n+1}_{h}).
\end{equation}
Consider $(\hat{u}^{n+1}_{h},u^{n+1}_{h})$ and $-\beta (\nabla \cdot u^{n}_{h}, \nabla \cdot u^{n+1}_{h})$.  Use the polarization identity on each term.  Then,
\begin{align}\label{stability:hat1}
(\hat{u}^{n+1}_{h},u^{n+1}_{h}) &= \frac{1}{2}\Big(\|\hat{u}^{n+1}_{h}\|^{2} + \|u^{n+1}_{h}\|^{2} - \|\hat{u}^{n+1}_{h} - u^{n+1}_{h}\|^{2}\Big),\\
-\beta (\nabla \cdot u^{n}_{h}, \nabla \cdot u^{n+1}_{h}) &= -\frac{\beta}{2}\Big( \|\nabla \cdot u^{n}_{h}\|^{2} + \| \nabla \cdot u^{n+1}_{h}\|^{2} - \| \nabla \cdot (u^{n+1}_{h} - u^{n}_{h}) \|^{2} \Big). \label{stability:timestepcoupling}
\end{align}
Use (\ref{stability:hat1}) and (\ref{stability:timestepcoupling}) in (\ref{stability:main}) and multiply by 2.  This yields
\begin{align*}
\|\hat{u}^{n+1}_{h}\|^{2} = \|u^{n+1}_{h}\|^{2} + \|\hat{u}^{n+1}_{h} - u^{n+1}_{h}\|^{2} + 2\gamma \Delta t \|\nabla \cdot u^{n+1}_{h}\|^{2} + \beta \Big(\|\nabla \cdot u^{n+1}_{h}\|^{2} - \|\nabla \cdot u^{n}_{h}\|^{2}
\\ + \| \nabla \cdot (u^{n+1}_{h} - u^{n}_{h}) \|^{2} \Big),
\end{align*}
as needed.  
\\ \indent For the second claim, consider equation (\ref{step2}), let $v_{h} = u^{n+1}_{h}$ and rearrange.  Then,
\begin{equation} \label{stability:mainb}
(\hat{u}^{n+1}_{h},u^{n+1}_{h}) = \|u^{n+1}_{h}\|^{2} + \gamma \Delta t gd_{lag}(u^{n+1}_{h},u^{n+1}_{h}).
\end{equation}
Use Lemma \ref{l1} and rearrange.  Then,
\begin{multline} 
(\hat{u}^{n+1}_{h},u^{n+1}_{h}) = \|u^{n+1}_{h}\|^{2} + \frac{\gamma \Delta t}{2} \Big(\|u^{n+1}_{1h,1}\|^{2} - \|u^{n}_{1h,1}\|^{2} + \|u^{n+1}_{2h,2}\|^{2} - \|u^{n}_{2h,2}\|^{2}\Big)
\\ + \frac{\gamma \Delta t}{2} \Big(\|u^{n+1}_{1h,1} + u^{n}_{2h,2} \|^{2} + \|u^{n+1}_{2h,2} + u^{n}_{1h,1} \|^{2}\Big).
\end{multline}
Using (\ref{stability:hat1}) in the above equation and multiplying by 2 yields the result.
\end{proof}
\end{lemma}

Next, we prove unconditional, nonlinear, long-time, energy stability of Algorithm 1.

\begin{theorem} \label{t1}
Consider Algorithm 1, (\ref{step1}) - (\ref{step2}).   Suppose $f \in L^{2}(0,t^{\ast};H^{-1}(\Omega)^{d})$, then
\begin{multline*}
\|u^{N}_{h}\|^{2} + \beta \|\nabla \cdot u^{N}_{h}\|^{2} + \sum_{n=0}^{N-1}\Big(\|\hat{u}^{n+1}_{h} - u^{n+1}_{h}\|^{2} + \|\hat{u}^{n+1}_{h} - u^{n}_{h}\|^{2} + \beta \| \nabla \cdot (u^{n+1}_{h} - u^{n}_{h})\|^{2} \Big)
\\ + 2\gamma \vertiii{\nabla \cdot u_{h}}^{2}_{2,0} + \nu \vertiii{\nabla \hat{u}_{h}}^{2}_{2,0} \leq \|u^{0}_{h}\|^{2} + \beta \|\nabla \cdot u^{0}_{h}\|^{2} + \frac{1}{\nu} \vertiii{f^{n+1}}^{2}_{2,-1}.
\end{multline*}
\end{theorem}
\begin{proof}
Let $(v_{h},q_{h}) = (2\Delta t\hat{u}^{n+1}_{h},2\Delta tp^{n+1}_{h})$ in equation (\ref{step1}), use skew-symmetry, the polarization identity, and rearrange.  Then,
\begin{align}
\|\hat{u}^{n+1}_{h}\|^{2} - \|u^{n}_{h}\|^{2} + \|\hat{u}^{n+1}_{h} - u^{n}_{h}\|^{2} + 2\nu \Delta t \|\nabla \hat{u}^{n+1}_{h}\|^{2} =  2\Delta t (f^{n+1},\hat{u}^{n+1}_{h}). \label{stability:main2}
\end{align}
Use Lemma \ref{l3} in (\ref{stability:main2}).  This yields
\begin{align} \label{stability:afterlemma}
\|u^{n+1}_{h}\|^{2} - \|u^{n}_{h}\|^{2} + \|\hat{u}^{n+1}_{h} - u^{n+1}_{h}\|^{2} + \|\hat{u}^{n+1}_{h} - u^{n}_{h}\|^{2}  + \beta \Big(\|\nabla \cdot u^{n+1}_{h}\|^{2} - \|\nabla \cdot u^{n}_{h}\|^{2}
\\ + \| \nabla \cdot (u^{n+1}_{h} - u^{n}_{h}) \|^{2} \Big) + 2\gamma \Delta t \|\nabla \cdot u^{n+1}_{h}\|^{2} + 2\nu \Delta t \|\nabla \hat{u}^{n+1}_{h}\|^{2} =  2\Delta t (f^{n+1},\hat{u}^{n+1}_{h}). \notag
\end{align}
Use the Cauchy-Schwarz-Young inequality on $\Delta t (f^{n+1},\hat{u}^{n+1}_{h})$,
\begin{align}
2\Delta t (f^{n+1},\hat{u}^{n+1}_{h}) &\leq \frac{\Delta t}{\nu} \|f^{n+1}\|^{2}_{-1} + \nu \Delta t \|\nabla \hat{u}^{n+1}_{h}\|^{2}. \label{stability:estf}
\end{align}
Use the above estimate in (\ref{stability:afterlemma}).  Then,
\begin{align*}
\|u^{n+1}_{h}\|^{2} - \|u^{n}_{h}\|^{2} + \|\hat{u}^{n+1}_{h} - u^{n+1}_{h}\|^{2} + \|\hat{u}^{n+1}_{h} - u^{n}_{h}\|^{2}  + \beta \Big(\|\nabla \cdot u^{n+1}_{h}\|^{2} - \|\nabla \cdot u^{n}_{h}\|^{2} + \| \nabla \cdot (u^{n+1}_{h} - u^{n}_{h}) \|^{2} \Big) 
\\ + 2\gamma \Delta t \|\nabla \cdot u^{n+1}_{h}\|^{2} + \nu \Delta t \|\nabla \hat{u}^{n+1}_{h}\|^{2} \leq \frac{\Delta t}{\nu} \|f^{n+1}\|^{2}_{-1}.
\end{align*}
\noindent Sum from $n = 0$ to $n = N-1$ and put all data on the r.h.s.  This yields
\begin{multline*}
\|u^{N}_{h}\|^{2} + \beta \|\nabla \cdot u^{N}_{h}\|^{2} + \sum_{n=0}^{N-1}\Big(\|\hat{u}^{n+1}_{h} - u^{n+1}_{h}\|^{2} + \|\hat{u}^{n+1}_{h} - u^{n}_{h}\|^{2} + \beta \| \nabla \cdot (u^{n+1}_{h} - u^{n}_{h})\|^{2} \Big)
\\ + 2\gamma \vertiii{\nabla \cdot u_{h}}^{2}_{2,0} + \nu \vertiii{\nabla \hat{u}_{h}}^{2}_{2,0} \leq \|u^{0}_{h}\|^{2} + \beta \|\nabla \cdot u^{0}_{h}\|^{2} + \frac{1}{\nu} \vertiii{f^{n+1}}^{2}_{2,-1}.
\end{multline*}
\noindent Therefore, the l.h.s. is bounded by data on the r.h.s. The velocity approximation is stable.
\end{proof}

Stability also holds in 2d for modular lagged grad-div in Algorithm 2.  As noted above, stability in 3d is an open problem.

\begin{theorem}[Lagged grad-div] \label{t2}
Consider Algorithm 2, (\ref{step1}) - (\ref{step2b}).   Suppose $f \in L^{2}(0,t^{\ast};H^{-1}(\Omega)^{d})$, then
\begin{multline*}
	\|u^{N}_{h}\|^{2} + \gamma \Delta t \Big(\|u^{N}_{1h,1}\|^{2} + \|u^{N}_{2h,2}\|^{2}\Big) + \sum_{n=0}^{N-1}\Big(\|\hat{u}^{n+1}_{h} - u^{n+1}_{h}\|^{2} + \|\hat{u}^{n+1}_{h} - u^{n}_{h}\|^{2}\Big)
	\\ + \gamma \Delta t \sum_{n=0}^{N-1}\Big(\|u^{n+1}_{1h,1} + u^{n}_{2h,2} \|^{2} + \|u^{n+1}_{2h,2} + u^{n}_{1h,1} \|^{2}\Big) + \nu \vertiii{\nabla \hat{u}_{h}}^{2}_{2,0}  \leq \|u^{0}_{h}\|^{2}
	\\ + \gamma \Delta t \Big(\|u^{0}_{1h,1}\|^{2} + \|u^{0}_{2h,2}\|^{2}\Big) + \frac{1}{\nu} \vertiii{f^{n+1}}^{2}_{2,-1}.
\end{multline*}
\end{theorem}
\begin{proof}
	Let $(v_{h},q_{h}) = (2\Delta t\hat{u}^{n+1}_{h},2\Delta tp^{n+1}_{h})$ in equation (\ref{step1}) and similar techniques yields equation (\ref{stability:main2}).  Moreover, using Lemma \ref{l3} in (\ref{stability:main2}) yields
\begin{multline} \label{stability:afterlemma2}
\|u^{n+1}_{h}\|^{2} - \|u^{n}_{h}\|^{2} + \|\hat{u}^{n+1}_{h} - u^{n+1}_{h}\|^{2} + \|\hat{u}^{n+1}_{h} - u^{n}_{h}\|^{2} 
\\ + \gamma \Delta t \Big(\|u^{n+1}_{1h,1}\|^{2} - \|u^{n}_{1h,1}\|^{2} + \|u^{n+1}_{2h,2}\|^{2} - \|u^{n}_{2h,2}\|^{2}\Big)
\\ + \gamma \Delta t \Big(\|u^{n+1}_{1h,1} + u^{n}_{2h,2} \|^{2} + \|u^{n+1}_{2h,2} + u^{n}_{1h,1} \|^{2}\Big) + 2\nu \Delta t \|\nabla \hat{u}^{n+1}_{h}\|^{2} =  2\Delta t (f^{n+1},\hat{u}^{n+1}_{h}).
\end{multline}
Use estimate (\ref{stability:estf}) in (\ref{stability:afterlemma2}).  Then,
\begin{align*}
\|u^{n+1}_{h}\|^{2} - \|u^{n}_{h}\|^{2} + \|\hat{u}^{n+1}_{h} - u^{n+1}_{h}\|^{2} + \|\hat{u}^{n+1}_{h} - u^{n}_{h}\|^{2} + \gamma \Delta t \Big(\|u^{n+1}_{1h,1}\|^{2} - \|u^{n}_{1h,1}\|^{2} + \|u^{n+1}_{2h,2}\|^{2} - \|u^{n}_{2h,2}\|^{2}\Big)
\\ + \gamma \Delta t \Big(\|u^{n+1}_{1h,1} + u^{n}_{2h,2} \|^{2} + \|u^{n+1}_{2h,2} + u^{n}_{1h,1} \|^{2}\Big) + \nu \Delta t \|\nabla \hat{u}^{n+1}_{h}\|^{2} \leq \frac{\Delta t}{\nu} \|f^{n+1}\|^{2}_{-1}.
\end{align*}
\noindent Summing from $n = 0$ to $n = N-1$ and putting all data on the r.h.s. yields the result. Therefore, the l.h.s. is bounded by data on the r.h.s. The velocity approximation is stable.
\end{proof}

Theorem \ref{t1} shows that the natural kinetic energy and energy dissipation rates of Algorithm 1 are, respectively,
\begin{align*}
Kinetic \; energy (t^{N}) &= \frac{1}{2} \|u^{N}_{h}\|^{2} + \frac{1}{2} \beta \| \nabla \cdot u^{N}_{h} \|^{2}, \\
Energy \; dissipation \; rate (t^{n}) &= \frac{\Delta t}{2} \|\frac{\hat{u}^{n+1}_{h} - u^{n+1}_{h}}{\Delta t}\|^{2} + \frac{\Delta t}{2} \| \frac{\hat{u}^{n+1}_{h} - u^{n}_{h}}{\Delta t}\|^{2} + \frac{\beta \Delta t}{2} \| \nabla \cdot \big(\frac{u^{n+1}_{h} - u^{n}_{h}}{\Delta t}\big)\|^{2}
\\ &+ \gamma \| \nabla \cdot u^{n+1}_{h}\|^{2} + \frac{\nu}{2}\| \nabla \hat{u}^{n+1}_{h} \|^{2}.
\end{align*}
The term $\frac{1}{2} \beta \| \nabla \cdot u^{N}_{h} \|^{2}$ in the kinetic energy arises from the dispersive regularization $-\beta \nabla \nabla \cdot u_{t}$.  In the energy dissipation rate, the second and fifth terms are the standard numerical and molecular dissipation.  The other terms penalize violations of incompressibility.  A similar interpretation holds for the terms in Theorem \ref{t2} for Algorithm 2.
\subsection{Error Analysis}
We now analyze the method's consistency error and prove convergence at the expected rates.  Denote $u^{n}$ and $p^{n}$ as the true solutions at time $t^{n} = n\Delta t$.  Assume the solutions satisfy the following regularity assumptions:
\begin{align} 
u \in L^{\infty}(0,t^{\ast};X \cap H^{k+1}(\Omega)^{d}), \; u_{t} &\in L^{\infty}(0,t^{\ast};H^{k+1}(\Omega)^{d}),\label{error:regularity}
\\ u_{tt} \in L^{2}(0,t^{\ast};L^{2}(\Omega)^{d}), \; p &\in L^{2}(0,t^{\ast};Q \cap H^{m}(\Omega)). \label{error:regularity2}
\end{align}
The errors for the solution variables are denoted
\begin{align*}
e^{n}_{u} &= u^{n} - u^{n}_{h}, \; e^{n}_{\hat{u}} = \hat{u}^{n} - u^{n}_{h}, \; e^{n}_{p} = p^{n} - p^{n}_{h}.
\end{align*}\vspace*{-\baselineskip}
\begin{definition} (Consistency errors).  The consistency errors are defined as
\begin{align*}
\tau_{1}(u^{n};v_{h}) &= \big(\frac{u^{n}-u^{n-1}}{\Delta t} - u^{n}_{t}, v_{h}\big) - b(u^{n} - u^{n-1},u^{n},v_{h}),
\\ \tau_{2}(u^{n};v_{h}) &= (u^{n}_{1,1}-u^{n-1}_{1,1},v_{2h,2}) + (u^{n}_{2,2}-u^{n-1}_{2,2},v_{1h,1}).
\end{align*}
\end{definition}
\begin{lemma}\label{consistency}
Provided $u$ satisfies the regularity assumptions \ref{error:regularity} - \ref{error:regularity2}, then $\forall \epsilon, r > 0$
\begin{align*}
\lvert \tau_{1}(u^{n};v_{h}) \rvert &\leq \frac{C\Delta t}{\epsilon} \Big(\| u_{tt}\|^{2}_{L^{2}(t^{n-1},t^{n};L^{2}(\Omega))} + \| \nabla u_{t}\|^{2}_{L^{2}(t^{n-1},t^{n};L^{2}(\Omega))}\Big) + \frac{\epsilon}{r} \| \nabla v_{h} \|^{2},
\\ \lvert \tau_{2}(u^{n};v_{h}) \rvert &\leq \frac{C\Delta t^{2}}{\epsilon} \|\nabla u_{t}\|^{2}_{L^{\infty}(t^{n-1},t^{n};L^{2}(\Omega))} + \frac{\epsilon}{r} \Big(\| v_{1h,1} \|^{2} + \| v_{2h,2} \|^{2}\Big).
\end{align*}
\end{lemma}
\begin{proof}
The first follows from Lemma \ref{l1}, the Cauchy-Schwarz-Young inequality, Poincar\'{e}-Friedrichs inequality, and Taylor's Theorem with integral remainder.  The second is similar.
\end{proof}
Once again, a key lemma on the effect of Step 2 on the evolution of the error is critical in the error analysis.
\begin{lemma}\label{l6}
Consider equation (\ref{step2}).  Let $e^{n}_{\hat{u}} = (u^{n} - \tilde{u}^{n}) - (\hat{u}^{n}_{h}- \tilde{u}^{n}) = \eta^{n} - \psi^{n}_{h}$.  Similarly, $e^{n}_{u} = (u^{n} - \tilde{u}^{n}) - (u^{n}_{h}- \tilde{u}^{n}) = \eta^{n} - \phi^{n}_{h}$.  The following inequality holds,
\begin{align*}
\| \psi^{n+1}_{h} \|^{2} \geq \|\phi^{n+1}_{h}\|^{2} + \|\phi^{n+1}_{h}-\psi^{n+1}_{h}\|^{2} + \beta \Big( \|\nabla \cdot \phi^{n+1}_{h}\|^{2} - \|\nabla \cdot \phi^{n}_{h}\|^{2} + \frac{1}{2}\|\nabla \cdot (\phi^{n+1}_{h}-\phi^{n}_{h})\|^{2} \Big)
\\ + \gamma \Delta t \| \nabla \cdot \phi^{n+1}_{h} \|^{2} - \beta \Delta t \| \nabla \cdot \phi^{n}_{h} \|^{2} - d\beta (1 + 2\Delta t) \| \nabla \eta_{t} \|^{2}_{L^{2}(t^{n},t^{n+1};L^{2}(\Omega))} - d \gamma \Delta t \| \nabla \eta^{n+1} \|^{2}.
\end{align*}
Moreover, considering equation (\ref{step2b}),
\begin{multline*}
	\|\psi^{n+1}_{h}\|^{2} \geq \|\phi^{n+1}_{h}\|^{2} + \|\phi^{n+1}_{h}-\psi^{n+1}_{h}\|^{2} + \gamma \Delta t \Big( \|\phi^{n+1}_{1h,1}\|^{2} - \|\phi^{n}_{1h,1}\|^{2} + \|\phi^{n+1}_{2h,2}\|^{2} - \|\phi^{n}_{2h,2}\|^{2}\Big)
	\\ + \frac{\gamma \Delta t}{4} \Big(\|\phi^{n+1}_{1h,1}+\phi^{n}_{1h,1}\|^{2} + \|\phi^{n+1}_{2h,2}+\phi^{n}_{2h,2}\|^{2}\Big) - \gamma(1 + 8\Delta t) \Big(\| \nabla \eta^{n+1}\|^{2} + \|\nabla \eta^{n}\|^{2}\Big)
\\ - C \gamma \Delta t^{2} \Big(\|\nabla u_{t} \|^{2}_{L^{2}(t^{n},t^{n+1};L^{2}(\Omega))} + \|\nabla u_{t} \|^{2}_{L^{\infty}(t^{n},t^{n+1};L^{2}(\Omega))}\Big) - 2 \gamma \Delta t^{2} \Big(\|\phi^{n}_{1h,1}\|^{2} + \|\phi^{n}_{2h,2}\|^{2}\Big)
\end{multline*}
\end{lemma}
\begin{proof}
The true solution satisfies for all $n = 0, 1, ..., N-1$:
\begin{align} \label{true:step2}
(u^{n+1},v_{h}) + \beta (\nabla \cdot (u^{n+1} - u^{n}), \nabla \cdot v_{h}) + \gamma \Delta t (\nabla \cdot u^{n+1}, \nabla \cdot v_{h}) = (u^{n+1},v_{h}) \; \; \forall v_{h} \in X_{h}.
\end{align}
Subtract (\ref{step2}) from (\ref{true:step2}), then the error equation is
\begin{align*}
(e^{n+1}_{u},v_{h}) + \beta (\nabla \cdot (e^{n+1}_{u} - e^{n}_{u}), \nabla \cdot v_{h}) + \gamma \Delta t (\nabla \cdot e^{n+1}_{u}, \nabla \cdot v_{h}) = (e^{n+1}_{\hat{u}},v_{h}) \; \; \forall v_{h} \in X_{h}.
\end{align*}
Let $v_{h} = \phi^{n+1}_{h}$.  Decompose the error terms and reorganize.  Then,
\begin{align}\label{keyeq}
\|\phi^{n+1}_{h}\|^{2} + \frac{\beta}{2} \Big( \|\nabla \cdot \phi^{n+1}_{h}\|^{2} - \|\nabla \cdot \phi^{n}_{h}\|^{2} + \|\nabla \cdot (\phi^{n+1}_{h}-\phi^{n}_{h})\|^{2} \Big) + \gamma \Delta t \| \nabla \cdot \phi^{n+1}_{h} \|^{2}
\\ = \beta (\nabla \cdot (\eta^{n+1}-\eta^{n}),\nabla \cdot \phi^{n+1}_{h}) + \gamma \Delta t (\nabla \cdot \eta^{n+1}, \nabla \cdot \phi^{n+1}_{h}) + (\psi^{n+1}_{h},\phi^{n+1}_{h}) \notag
\end{align}
Add and subtract $\beta (\nabla \cdot (\eta^{n+1}-\eta^{n}),\nabla \cdot \phi^{n}_{h})$ to the r.h.s. of equation (\ref{keyeq}).  Use the Cauchy-Schwarz-Young inequality on the following terms, as well as Taylor's Theorem with integral remainder on the first two terms,
\begin{align}
\beta (\nabla \cdot (\eta^{n+1}-\eta^{n}),\nabla \cdot (\phi^{n+1}_{h} - \phi^{n}_{h})) &\leq d \beta \Delta t \| \nabla \eta_{t} \|^{2}_{L^{2}(t^{n},t^{n+1};L^{2}(\Omega))} + \frac{\beta}{4} \| \nabla \cdot (\phi^{n+1}_{h} - \phi^{n}_{h}) \|^{2},
\\ \beta (\nabla \cdot (\eta^{n+1}-\eta^{n}),\nabla \cdot \phi^{n}_{h}) &\leq \frac{d\beta}{2} \| \nabla \eta_{t} \|^{2}_{L^{2}(t^{n},t^{n+1};L^{2}(\Omega))} + \frac{\beta \Delta t}{2} \| \nabla \cdot \phi^{n}_{h} \|^{2},
\\ \gamma \Delta t (\nabla \cdot \eta^{n+1}, \nabla \cdot \phi^{n+1}_{h}) &\leq \frac{d \gamma \Delta t}{2} \| \nabla \eta^{n+1} \|^{2} + \frac{\gamma \Delta t}{2} \| \nabla \cdot \phi^{n+1}_{h} \|^{2}.
\end{align}
Use the above inequalities and the polarization identity on $(\psi^{n+1}_{h},\phi^{n+1}_{h})$ in equation (\ref{keyeq}) and multiply by 2.  Then,
\begin{align*}
\|\phi^{n+1}_{h}\|^{2} + \|\phi^{n+1}_{h}-\psi^{n+1}_{h}\|^{2} + \beta \Big( \|\nabla \cdot \phi^{n+1}_{h}\|^{2} - \|\nabla \cdot \phi^{n}_{h}\|^{2} + \frac{1}{2}\|\nabla \cdot (\phi^{n+1}_{h}-\phi^{n}_{h})\|^{2} \Big) + \gamma \Delta t \| \nabla \cdot \phi^{n+1}_{h} \|^{2}
\\ \leq \| \psi^{n+1}_{h} \|^{2} + \beta \Delta t \| \nabla \cdot \phi^{n}_{h} \|^{2} + d\beta (1 + 2\Delta t) \| \nabla \eta_{t} \|^{2}_{L^{2}(t^{n},t^{n+1};L^{2}(\Omega))} + d \gamma \Delta t \| \nabla \eta^{n+1} \|^{2}.
\end{align*}
Reorganizing yields the first result.  For the second, notice that the true solution satisfies for all $n = 0, 1, ..., N-1$:
\begin{align} \label{true:step2b}
	(u^{n+1},v_{h}) + \gamma \Delta t \Big(gd_{lag}(u^{n+1},v_{h}) + \tau_{2}(u^{n+1},v_{h})\Big) = (u^{n+1},v_{h}) \; \; \forall v_{h} \in X_{h}.
\end{align}
Subtract (\ref{step2b}) from (\ref{true:step2b}), then the error equation is
\begin{align*}
	(e^{n+1}_{u},v_{h}) + \gamma \Delta t\Big(gd_{lag}(e^{n+1}_{u},v_{h}) + \tau_{2}(u^{n+1},v_{h})\Big) = (e^{n+1}_{\hat{u}},v_{h}) \; \; \forall v_{h} \in X_{h}.
\end{align*}
Let $v_{h} = \phi^{n+1}_{h}$.  Decompose the error terms, use Lemma \ref{l2} and reorganize.  Then,
\begin{multline}\label{keyeqb}
	\|\phi^{n+1}_{h}\|^{2} + \frac{\gamma \Delta t}{2} \Big( \|\phi^{n+1}_{1h,1}\|^{2} - \|\phi^{n}_{1h,1}\|^{2} + \|\phi^{n+1}_{2h,2}\|^{2} - \|\phi^{n}_{2h,2}\|^{2} + \|\phi^{n+1}_{1h,1}+\phi^{n}_{1h,1}\|^{2} + \|\phi^{n+1}_{2h,2}+\phi^{n}_{2h,2}\|^{2}\Big)
	\\ = \gamma \Delta t gd_{lag}(\eta^{n+1},\phi^{n+1}_{h}) + \gamma \Delta t \tau_{2}(u^{n+1},\phi^{n+1}_{h}) + (\psi^{n+1}_{h},\phi^{n+1}_{h}).
\end{multline}
Add and subtract $\gamma \Delta t gd_{lag}(\eta^{n+1},\phi^{n}_{h})$ and $\gamma \Delta t \tau_{2}(u^{n+1},\phi^{n}_{h})$ on the r.h.s.  Once again, use the Cauchy-Schwarz-Young inequality on the following terms as well as Lemma \ref{consistency} on the third and fourth,
\begin{align}
\gamma \Delta t gd_{lag}(\eta^{n+1},\phi^{n+1}_{h} + \phi^{n}_{h}) &\leq 8\gamma \Delta t \Big(\| \nabla \eta^{n+1}\|^{2} + \|\nabla \eta^{n})\|^{2}\Big)
\\ &+ \frac{\gamma \Delta t}{4} \Big(\| \phi^{n+1}_{1h,1} + \phi^{n}_{1h,1}\|^{2} + \| \phi^{n+1}_{2h,2} + \phi^{n}_{2h,2}\|^{2}\Big) \notag
\\-\gamma \Delta t gd_{lag}(\eta^{n+1},\phi^{n}_{h}) &\leq \gamma \Big(\| \nabla \eta^{n+1}\|^{2} + \|\nabla \eta^{n})\|^{2}\Big)
\\ &+ \frac{\gamma \Delta t^{2}}{2} \Big(\|\phi^{n}_{1h,1}\|^{2} + \|\phi^{n}_{2h,2}\|^{2}\Big), \notag
\\ \gamma \Delta t \tau_{2}(u^{n+1},\phi^{n+1}_{h}+\phi^{n}_{h}) &\leq {C \gamma \Delta t^{2}} \|\nabla u_{t} \|^{2}_{L^{2}(t^{n},t^{n+1};L^{2}(\Omega))}
\\ &+ \frac{\gamma \Delta t}{8} \Big(\| \phi^{n+1}_{1h,1} + \phi^{n}_{1h,1}\|^{2} + \| \phi^{n+1}_{2h,2} + \phi^{n}_{2h,2} \|^{2}\Big) \notag
\\ -\gamma \Delta t \tau_{2}(u^{n+1},\phi^{n}_{h}) &\leq \frac{C \gamma \Delta t^{2}}{2} \|\nabla u_{t} \|^{2}_{L^{\infty}(t^{n},t^{n+1};L^{2}(\Omega))}
\\ &+ \frac{\gamma \Delta t^{2}}{2} \Big(\|\phi^{n}_{1h,1}\|^{2} + \|\phi^{n}_{2h,2}\|^{2}\Big). \notag
\end{align}
Using the above inequalities and the polarization identity on $(\psi^{n+1}_{h},\phi^{n+1}_{h})$ in equation (\ref{keyeqb}) and multiplying by 2 yields, as claimed, that
\begin{multline*}
	\|\phi^{n+1}_{h}\|^{2} + \|\phi^{n+1}_{h}-\psi^{n+1}_{h}\|^{2} + \gamma \Delta t \Big( \|\phi^{n+1}_{1h,1}\|^{2} - \|\phi^{n}_{1h,1}\|^{2} + \|\phi^{n+1}_{2h,2}\|^{2} - \|\phi^{n}_{2h,2}\|^{2}\Big)
	\\ + \frac{\gamma \Delta t}{4} \Big(\|\phi^{n+1}_{1h,1}+\phi^{n}_{1h,1}\|^{2} + \|\phi^{n+1}_{2h,2}+\phi^{n}_{2h,2}\|^{2}\Big) \leq \|\psi^{n+1}_{h}\|^{2} + \gamma (1+8\Delta t) \Big(\| \nabla \eta^{n+1}\|^{2} + \|\nabla \eta^{n}\|^{2}\Big)
\\ + C \gamma \Delta t^{2} \Big(\|\nabla u_{t} \|^{2}_{L^{2}(t^{n},t^{n+1};L^{2}(\Omega))} + \|\nabla u_{t} \|^{2}_{L^{\infty}(t^{n},t^{n+1};L^{2}(\Omega))}\Big) + 2 \gamma \Delta t^{2} \Big(\|\phi^{n}_{1h,1}\|^{2} + \|\phi^{n}_{2h,2}\|^{2}\Big).
\end{multline*}
\end{proof}
We now prove convergence when $\beta>0$.
\begin{theorem} \label{t3}
For u and p satisfying (\ref{s1}), suppose that $u^{0}_{h} \in X_{h}$ is an approximation of $u^{0}$ to within the accuracy of the interpolant.  Suppose $\beta>0$, then there exists a constant $C>0$ such that the scheme (\ref{step1}) - (\ref{step2}) satisfies
\begin{align*}
\|e^{N}_{u}\|^{2} + \beta \|\nabla \cdot e^{N}_{u}\|^{2} + \sum_{n = 0}^{N-1} \Big(\|e^{n+1}_{\hat{u}} - e^{n}_{u}\|^{2} + \|e^{n+1}_{\hat{u}} - e^{n}_{u}\|^{2} + \beta \|\nabla \cdot (e^{n+1}_{u}-e^{n}_{u})\|^{2} \Big) + \gamma \vertiii{\nabla \cdot e_{u}}^{2}_{2,0}
\\ + \nu \vertiii{\nabla e_{\hat{u}}}^{2}_{2,0} \leq C \exp(C_{\dagger}t^{\ast}) \Big\{ \inf_{v_{h} \in X_{h}} \Big( \nu^{-1}\vertiii{(u - v_{h})_{t}}^{2}_{L^{2}(0,t^{\ast};H^{-1}(\Omega))}
\\ + (\nu^{-1} + \Delta t h^{-1}\nu^{-1} + d\gamma) \vertiii{\nabla(u - v_{h})}^{2}_{2,0} + d\beta(1+2\Delta t) \|\nabla(u - v_{h})_{t}\|^{2}_{L^{2}(0,t^{\ast};L^{2}(\Omega))}\Big)
\\ + d\nu^{-1}\inf_{q_{h} \in Q_{h}} \vertiii{ p - q_{h} }^{2}_{2,0} + \nu^{-1}\Delta t^{2} + \|e^{0}_{u}\|^{2} + \beta \|\nabla \cdot e^{0}_{u}\|^{2}\Big\}.
\end{align*}
\end{theorem}
\begin{proof}
Consider the scheme (\ref{step1}) - (\ref{step2}).  The true solution satisfies for all $n = 0, 1, ..., N-1$:
\begin{multline} \label{true:step1}
(\frac{u^{n+1} - u^{n}}{\Delta t},v_{h}) + b(u^{n},u^{n+1},v_{h}) + \nu (\nabla u^{n+1},\nabla v_{h}) - (p^{n+1},\nabla \cdot v_{h}) = (f^{n+1},v_{h})
\\ + \tau_{1}(u^{n+1};v_{h}) \; \; \forall v_{h} \in X_{h}.
\end{multline}
Subtract (\ref{step1}) from (\ref{true:step1}), then the error equation is
\begin{align*}
(\frac{e^{n+1}_{\hat{u}} - e^{n}_{u}}{\Delta t},v_{h}) + b(u^{n},u^{n+1},v_{h}) - b(u^{n}_{h},\hat{u}^{n+1},v_{h}) + \nu (\nabla e^{n+1}_{\hat{u}},\nabla v_{h}) - (e^{n+1}_{p},\nabla \cdot v_{h})
\\ = \tau_{1}(u^{n+1},v_{h}) \; \; \forall v_{h} \in X_{h}.
\end{align*}
Set $v_{h} = 2\Delta t\psi^{n+1}_{h} \in V_{h}$ and reorganize.  This yields
\begin{multline} \label{feu}
\|\psi^{n+1}_{h}\|^{2} - \|\phi^{n}_{h}\|^{2} + \|\psi^{n+1}_{h} - \phi^{n}_{h}\|^{2} + 2\nu \Delta t \| \nabla \psi^{n+1}_{h}\|^{2} = 2 (\eta^{n+1}-\eta^{n},\psi^{n+1}_{h}) + 2\nu \Delta t (\nabla \eta^{n+1},\nabla \psi^{n+1}_{h})
\\ + 2\Delta t b(u^{n},u^{n+1},\psi^{n+1}_{h}) - 2\Delta t b(u^{n}_{h},\hat{u}^{n+1},\psi^{n+1}_{h}) - 2\Delta t(p^{n+1}-q_{h},\nabla \cdot \psi^{n+1}_{h}) - 2\Delta t \tau_{1}(u^{n+1},\psi^{n+1}_{h}).
\end{multline}
Add and subtract $2\Delta t b(u^{n}_{h},u^{n+1},\psi^{n+1}_{h})$ and $2\Delta t b(\hat{u}^{n+1}_{h},\eta^{n+1},\psi^{n+1}_{h})$ to the r.h.s., use skew-symmetry, and reorganize.  Then,
\begin{multline} \label{feu1}
\|\psi^{n+1}_{h}\|^{2} - \|\phi^{n}_{h}\|^{2} + \|\psi^{n+1}_{h} - \phi^{n}_{h}\|^{2} + 2\nu \Delta t \| \nabla \psi^{n+1}_{h}\|^{2} = 2 (\eta^{n+1}-\eta^{n},\psi^{n+1}_{h}) + 2\nu \Delta t (\nabla \eta^{n+1},\nabla \psi^{n+1}_{h})
\\ + 2\Delta t b(\eta^{n},u^{n+1},\psi^{n+1}_{h}) - 2\Delta t b(\phi^{n}_{h},u^{n+1},\psi^{n+1}_{h}) + 2\Delta t b(\hat{u}^{n+1}_{h},\eta^{n+1},\psi^{n+1}_{h})
\\ + 2\Delta t b(u^{n}_{h}-\hat{u}^{n+1}_{h},\eta^{n+1},\psi^{n+1}_{h}) - 2\Delta t(p^{n+1}-q_{h},\nabla \cdot \psi^{n+1}_{h}) - 2\Delta t \tau_{1}(u^{n+1},\psi^{n+1}_{h}).
\end{multline}
The following estimates follow from application of the Cauchy-Schwarz-Young inequality,
\begin{align}
2(\eta^{n+1}-\eta^{n},\psi^{n+1}_{h}) &\leq \frac{4C_{r}}{\epsilon_1} \| \eta_{t} \|^{2}_{L^{2}(t^{n},t^{n+1};H^{-1}(\Omega))} + \frac{\epsilon_1 \Delta t}{r}\| \nabla \psi^{n+1}_{h} \|^{2}, \label{term1}
\\ 2\nu \Delta t (\nabla \eta^{n+1},\nabla \psi^{n+1}_{h}) &\leq \frac{4C_{r}\nu\Delta t}{\epsilon_2} \|\nabla \eta^{n+1}\|^{2} + \frac{\epsilon_2 \nu \Delta t}{r} \| \nabla \psi^{n+1}_{h} \|^{2}, \label{term2}
\\ - 2\Delta t(p^{n+1}_{h}-q_{h},\nabla \cdot \psi^{n+1}_{h}) &\leq \frac{4C_{r}d\Delta t}{\epsilon_3} \|p^{n+1}-q_{h}\|^{2} + \frac{\epsilon_3\Delta t}{r} \| \nabla \psi^{n+1}_{h} \|^{2}. \label{term3}
\end{align}
Applying Lemma \ref{l1} and the Cauchy-Schwarz-Young inequality yields,
\begin{align}
2\Delta t b(\eta^{n},u^{n+1},\psi^{n+1}_{h}) &\leq \frac{4C_{1}^{2}C_{r}\Delta t}{\epsilon_{4}} \|\nabla u^{n+1}\|^{2}\|\nabla \eta^{n}\|^{2} + \frac{\epsilon_{4}\Delta t}{r} \|\nabla \psi^{n+1}_{h}\|^{2}, \label{term4}
\\ -2\Delta t b(\phi^{n}_{h},u^{n+1},\psi^{n+1}_{h}) &\leq \frac{4C_{3}^{2}C_{r}\Delta t}{\epsilon_{5}}\|u^{n+1}\|^{2}_{2} \|\phi^{n}_{h}\|^{2} + \frac{\epsilon_{5}\Delta t}{r} \|\nabla \psi^{n+1}_{h}\|^{2} \label{term5}
\\ &+ \frac{4C_{3}^{2}C_{r}\Delta t}{\epsilon_{6}}\|\nabla u^{n+1}\|^{2} \|\nabla \cdot \phi^{n}_{h}\|^{2}+ \frac{\epsilon_{6}\Delta t}{r} \|\nabla \psi^{n+1}_{h}\|^{2},\notag
\\ 2\Delta t b(\hat{u}^{n+1}_{h},\eta^{n+1},\psi^{n+1}_{h}) &\leq 2C_{1}\Delta t \|\nabla \hat{u}^{n+1}_{h}\|\|\nabla \eta^{n+1}\|\|\nabla \psi^{n+1}_{h}\| \label{term6}
\\ &\leq \frac{4C^{2}_{1}C_{r}\Delta t}{\epsilon_{7}} \|\nabla \hat{u}^{n+1}_{h}\|^{2}\|\nabla \eta^{n+1}\|^{2} + \frac{\epsilon_{7}\Delta t}{r}\|\nabla \psi^{n+1}_{h}\|^{2}, \notag
\\ 2\Delta t b(u^{n}_{h}-\hat{u}^{n+1}_{h},\eta^{n+1},\psi^{n+1}_{h}) &\leq 2C_{2}\Delta t \sqrt{\|u^{n}_{h}-\hat{u}^{n+1}_{h}\|\|\nabla (u^{n}_{h}-\hat{u}^{n+1}_{h})\|}\|\nabla \eta^{n+1}\|\|\nabla \psi^{n+1}_{h}\| \label{term7}
\\ &\leq 2C_{2} C^{1/2}_{inv}\Delta th^{-1/2} \|u^{n}_{h}-\hat{u}^{n+1}_{h}\|\|\nabla \eta^{n+1}\|\|\nabla \psi^{n+1}_{h}\| \notag
\\ &\leq \frac{4C^{2}_{2}C_{r}C_{inv}\Delta t}{h\epsilon_{8}} \|u^{n}_{h}-\hat{u}^{n+1}_{h}\|^{2}\|\nabla \eta^{n+1}\|^{2} + \frac{\epsilon_{8}\Delta t}{r}\|\nabla \psi^{n+1}_{h}\|^{2}. \notag
\end{align}
Use Lemma \ref{l6} in equation (\ref{feu1}) and rearrange.  Then,
\begin{multline}
\|\phi^{n+1}_{h}\|^{2} - \|\phi^{n}_{h}\|^{2} + \|\psi^{n+1}_{h} - \phi^{n}_{h}\|^{2} + \|\phi^{n+1}_{h} - \phi^{n}_{h}\|^{2} + \beta \Big( \|\nabla \cdot \phi^{n+1}_{h}\|^{2} - \|\nabla \cdot \phi^{n}_{h}\|^{2}\Big) 
\\ + \frac{\beta}{2} \|\nabla \cdot (\phi^{n+1}_{h}-\phi^{n}_{h})\|^{2} + \gamma \Delta t \| \nabla \cdot \phi^{n+1}_{h} \|^{2} + 2\nu \Delta t \| \nabla \psi^{n+1}_{h}\|^{2} \leq 2 (\eta^{n+1}-\eta^{n},\psi^{n+1}_{h}) 
\\ + 2\nu \Delta t (\nabla \eta^{n+1},\nabla \psi^{n+1}_{h}) + 2\Delta t b(\eta^{n},u^{n+1},\psi^{n+1}_{h}) - 2\Delta t b(\phi^{n}_{h},u^{n+1},\psi^{n+1}_{h}) + 2\Delta t b(\hat{u}^{n+1}_{h},\eta^{n+1},\psi^{n+1}_{h})
\\ + 2\Delta t b(u^{n}_{h}-\hat{u}^{n+1}_{h},\eta^{n+1},\psi^{n+1}_{h}) - 2\Delta t \tau_{1}(u^{n+1},\psi^{n+1}_{h}) + \beta \Delta t \| \nabla \cdot \phi^{n}_{h} \|^{2}
\\ + d\beta(1+2\Delta t) \| \nabla \eta_{t} \|^{2}_{L^{2}(t^{n},t^{n+1};L^{2}(\Omega))} + d \gamma \Delta t \| \nabla \eta^{n+1} \|^{2}.
\end{multline}
Apply Lemma \ref{consistency}, let $r = 9$ and $\epsilon_1 = \nu \epsilon_2 = \epsilon_3 = \epsilon_4 = \epsilon_{5} = \epsilon_{6}  = \epsilon_{7} = \epsilon_{8} = \epsilon_{9} = \nu$.  Use the above estimates, and regroup:
\begin{multline}
\|\phi^{n+1}_{h}\|^{2} - \|\phi^{n}_{h}\|^{2} + \|\psi^{n+1}_{h} - \phi^{n}_{h}\|^{2} + \|\phi^{n+1}_{h} - \phi^{n}_{h}\|^{2} + \beta \Big( \|\nabla \cdot \phi^{n+1}_{h}\|^{2} - \|\nabla \cdot \phi^{n}_{h}\|^{2}\Big) 
\\ + \beta \|\nabla \cdot (\phi^{n+1}_{h}-\phi^{n}_{h})\|^{2} + \gamma \Delta t \| \nabla \cdot \phi^{n+1}_{h} \|^{2} + \nu \Delta t \| \nabla \psi^{n+1}_{h}\|^{2} \leq \frac{4C_{r}}{\nu} \| \eta_{t} \|^{2}_{L^{2}(t^{n},t^{n+1};H^{-1}(\Omega))} + \frac{4C_{r}\Delta t}{\nu} \|\nabla \eta^{n+1}\|^{2}
\\ + \frac{4C_{1}^{2}C_{r}\Delta t}{\nu} \|\nabla u^{n+1}\|^{2}\|\nabla \eta^{n}\|^{2} + \frac{4C_{3}^{2}C_{r}\Delta t}{\nu} \Big(\|u^{n+1}\|^{2}_{2} \|\phi^{n}_{h}\|^{2} + (1 + \beta^{-1}\|\nabla u^{n+1}\|^{2} )\beta\|\nabla \cdot \phi^{n}_{h}\|^{2} \Big)
\\ + \frac{4C_{1}^{2}C_{r}\Delta t}{\nu}\|\nabla \hat{u}^{n+1}_{h}\|^{2}\|\nabla \eta^{n+1} \|^{2} + \frac{4C_{2}^{2}C_{r}C_{inv}\Delta t}{h\nu}\|\hat{u}^{n+1}_{h}-u^{n}_{h}\|^{2}\|\nabla \eta^{n+1} \|^{2} + \frac{4C_{r}d\Delta t}{\nu} \|p^{n+1}-q_{h}\|^{2}
\\ + \frac{C\Delta t^{2}}{\nu} \Big(\| u_{tt}\|^{2}_{L^{2}(t^{n},t^{n+1};L^{2}(\Omega))} + \| \nabla u_{t}\|^{2}_{L^{2}(t^{n},t^{n+1};L^{2}(\Omega))}\Big)
\\ + d\beta(1+2\Delta t) \| \nabla \eta_{t} \|^{2}_{L^{2}(t^{n},t^{n+1};L^{2}(\Omega))} + d \gamma \Delta t \| \nabla \eta^{n+1} \|^{2}.
\end{multline}
Denote $C_{\dagger} = 4C_{3}^{2}C_{r}\nu^{-1}\max\{1,\vertiii{u}^{2}_{\infty,2},\beta^{-1}\vertiii{\nabla u}^{2}_{\infty,0}\}$. Sum from $n = 0$ to $n = N-1$, use Theorem \ref{t1}, take maximums over constants pertaining to Gronwall terms and remaining terms, use Lemma \ref{l4}, take infimums over $V_{h}$ and $Q_{h}$, and use Lemma \ref{l5}.  Renorm, then,
\begin{align*}
\|\phi^{N}_{h}\|^{2} + \beta \|\nabla \cdot \phi^{N}_{h}\|^{2} + \sum_{n = 0}^{N-1} \Big(\|\phi^{n+1}_{h} - \psi^{n+1}_{h}\|^{2} + \|\psi^{n+1}_{h} - \phi^{n}_{h}\|^{2} + \beta \|\nabla \cdot (\phi^{n+1}_{h}-\phi^{n}_{h})\|^{2} \Big)
\\ + \gamma \vertiii{\nabla \cdot \phi_{h}}^{2}_{2,0}  + \nu \vertiii{\nabla \psi_{h}}^{2}_{2,0} \leq C \exp(C_{\dagger}t^{\ast}) \Big\{ \inf_{v_{h} \in X_{h}} \Big( \nu^{-1}\vertiii{(u - v_{h})_{t}}^{2}_{L^{2}(0,t^{\ast};H^{-1}(\Omega))}
\\ + (\nu^{-1} + \Delta t h^{-1}\nu^{-1} + d\gamma) \vertiii{\nabla(u - v_{h})}^{2}_{2,0} + d\beta(1+2\Delta t) \|\nabla(u - v_{h})_{t}\|^{2}_{L^{2}(0,t^{\ast};L^{2}(\Omega))}\Big)
\\ + d\nu^{-1}\inf_{q_{h} \in Q_{h}} \vertiii{ p - q_{h} }^{2}_{2,0} + \nu^{-1}\Delta t^{2} + \|\phi^{0}_{h}\|^{2} + \beta \|\nabla \cdot \phi^{0}_{h}\|^{2}\Big\} .
\end{align*}
Using $\|\psi^{0}_{h}\| = \|\nabla \cdot \psi^{0}_{h}\| = 0$ and applying the triangle inequality yields the result.  
\end{proof}
\begin{corollary}
	Suppose the assumptions of Theorem \ref{t1} hold with $k=m=2$.  Further suppose that the finite element spaces ($X_{h}$,$Q_{h}$) are given by P2-P1 (Taylor-Hood), then the errors in velocity satisfy
\begin{align*}
\|e^{N}_{u}\|^{2} + \beta \|\nabla \cdot e^{N}_{u}\|^{2} + \sum_{n = 0}^{N-1} \Big(\|e^{n+1}_{\hat{u}} - e^{n}_{u}\|^{2} + \|e^{n+1}_{\hat{u}} - e^{n}_{u}\|^{2} + \beta \|\nabla \cdot (e^{n+1}_{u}-e^{n}_{u})\|^{2} \Big) + \gamma \vertiii{ \nabla \cdot e_{u} }^{2}_{2,0}
\\ + \nu \vertiii{\nabla e_{\hat{u}}}^{2}_{2,0} \leq C\exp(C_{\dagger}t^{\ast}) \Big(h^{6} + h^{4}+ h^{3}\Delta t + h^{4} + \Delta t^{2} + \|e^{0}_{u}\|^{2} + \beta \|\nabla \cdot e^{0}_{u}\|^{2}\Big).
\end{align*}
\end{corollary}
\begin{corollary}
	Suppose the assumptions of Theorem \ref{t1} hold with $k=m=1$.  Further suppose that the finite element spaces ($X_{h}$,$Q_{h}$) are given by P1b-P1 (MINI element), then the errors in velocity satisfy
\begin{align*}
\|e^{N}_{u}\|^{2} + \beta \|\nabla \cdot e^{N}_{u}\|^{2} + \sum_{n = 0}^{N-1} \Big(\|e^{n+1}_{\hat{u}} - e^{n}_{u}\|^{2} + \|e^{n+1}_{\hat{u}} - e^{n}_{u}\|^{2} + \beta \|\nabla \cdot (e^{n+1}_{u}-e^{n}_{u})\|^{2} \Big) + \gamma \vertiii{ \nabla \cdot e_{u} }^{2}_{2,0}
\\ + \nu \vertiii{\nabla e_{\hat{u}}}^{2}_{2,0} \leq C\exp(C_{\dagger}t^{\ast}) \Big(h^{4} + h^{2} + h\Delta t + h^{2} + \Delta t^{2} + \|e^{0}_{u}\|^{2} + \beta \|\nabla \cdot e^{0}_{u}\|^{2}\Big).
\end{align*}
\end{corollary}
When $\beta = 0$ we have the following result.
\begin{theorem} \label{t4}
	For u and p satisfying (\ref{s1}), suppose that $u^{0}_{h} \in X_{h}$ is an approximation of $u^{0}$ to within the accuracy of the interpolant.  Suppose $\beta=0$ and $u \in L^{\infty}(0,t^{\ast};X \cap L^{\infty}(\Omega)^{d})$, then there exists a constant $C>0$ such that the scheme (\ref{step1}) - (\ref{step2}) satisfies
\begin{multline*}
\|e^{N}_{u}\|^{2} + \sum_{n = 0}^{N-1} \Big(\|e^{n+1}_{\hat{u}} - e^{n}_{u}\|^{2} + \|e^{n+1}_{\hat{u}} - e^{n}_{u}\|^{2}\Big) + \gamma \vertiii{\nabla \cdot e_{u}}^{2}_{2,0} + \nu \vertiii{\nabla e_{\hat{u}}}^{2}_{2,0} 
\\ \leq C \exp(C_{\#}t^{\ast}) \Big\{ \inf_{v_{h} \in X_{h}} \Big( \nu^{-1}\vertiii{(u - v_{h})_{t}}^{2}_{L^{2}(0,t^{\ast};H^{-1}(\Omega))} + (\nu^{-1} + \Delta t h^{-1}\nu^{-1} + d\gamma) \vertiii{\nabla(u - v_{h})}^{2}_{2,0}\Big)
	\\ + d\nu^{-1}\inf_{q_{h} \in Q_{h}} \vertiii{ p - q_{h} }^{2}_{2,0} + \nu^{-1}\Delta t^{2} + \|e^{0}_{u}\|^{2}\Big\}.
\end{multline*}
	Moreover, $\exists \; C > 0$ and $C_{\triangle} = \max\{2,C_{\#}\}$ such that the scheme (\ref{step1}) - (\ref{step2b}) satisfies
\begin{multline*}
\|e^{N}_{u}\|^{2} + \gamma \Delta t \Big( \|e^{N}_{1u,1}\|^{2} + \|e^{N}_{2u,2}\|^{2}\Big) + \sum_{n = 0}^{N-1} \Big(\|e^{n+1}_{\hat{u}} - e^{n}_{u}\|^{2} + \|e^{n+1}_{\hat{u}} - e^{n}_{u}\|^{2}\Big) + \nu \vertiii{\nabla e_{\hat{u}}}^{2}_{2,0} 
\\ + \frac{\gamma \Delta t}{4} \sum_{n = 0}^{N-1} \Big(\|e^{n+1}_{1u,1} + e^{n}_{1u,1}\|^{2} + \|e^{n+1}_{2u,2} + e^{n}_{2u,2}\|^{2} \Big) \leq C \exp(C_{\triangle}t^{\ast}) \Big\{ \inf_{v_{h} \in X_{h}} \Big( \nu^{-1}\vertiii{(u - v_{h})_{t}}^{2}_{L^{2}(0,t^{\ast};H^{-1}(\Omega))}
\\ + (\nu^{-1} + \Delta t h^{-1}\nu^{-1} + \gamma + \gamma \Delta t) \vertiii{\nabla(u - v_{h})}^{2}_{2,0}\Big) + d\nu^{-1}\inf_{q_{h} \in Q_{h}} \vertiii{ p - q_{h} }^{2}_{2,0} + \gamma\Delta t^{2} + \|e^{0}_{u}\|^{2}
\\ + \gamma \Delta t \Big( \|e^{0}_{1u,1}\|^{2} + \|e^{0}_{2u,2}\|^{2}\Big)\Big\}.
	\end{multline*}
\end{theorem}
\begin{proof}
We prove only the first, the second is nearly identical.  We begin from equation (\ref{feu1}) and consider $2\Delta t b(\hat{u}^{n+1}_{h},\eta^{n+1},\psi^{n+1}_{h})$.  Applying Lemma 1 (inequality 4) and the Cauchy-Schwarz-Young inequality yields
\begin{align}\label{term5b}
-2\Delta t b(\phi^{n}_{h},u^{n+1},\psi^{n+1}_{h}) &\leq \frac{4C_{r}\Delta t}{\epsilon_{5}}\big(C_{3}^{2}\|u^{n+1}\|^{2}_{2} + \|u^{n+1}\|^{2}_{\infty}\big) \|\phi^{n}_{h}\|^{2} + \frac{\epsilon_{5}\Delta t}{r} \|\nabla \psi^{n+1}_{h}\|^{2}. 
\end{align}
Let $C_{\#} = 4C_{r}\nu^{-1}\max\{C_{3}^{2}\vertiii{u}^{2}_{\infty,2},\max_{0 \leq n \leq N-1} \|u^{n+1}\|^{2}_{\infty}\}$.  Use Lemma \ref{l6} in equation (\ref{feu1}), estimates (\ref{term1}) - (\ref{term4}), (\ref{term6}) - (\ref{term7}), and (\ref{term5b}), and Lemma \ref{consistency}.  Similar techniques used in Theorem \ref{t3} yield the result.
\end{proof}
 \begin{figure}
 	\centering 	\includegraphics[width=6in,height=\textheight, keepaspectratio]{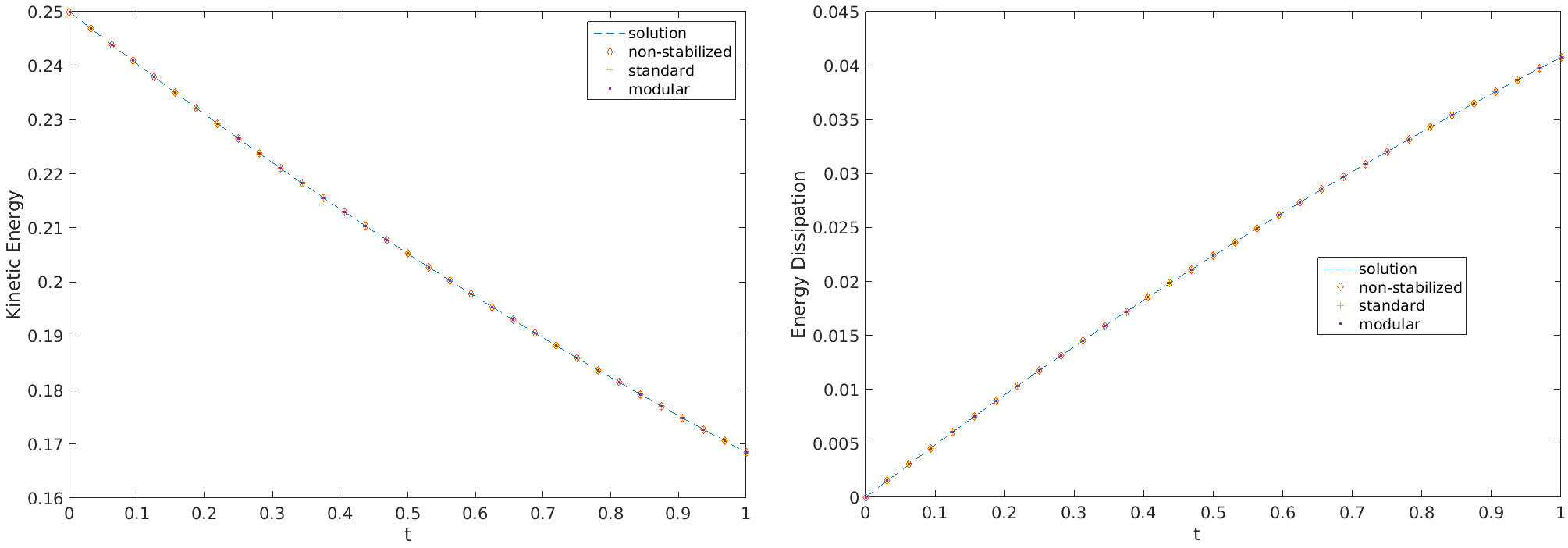}
 	\caption{Kinetic energy (left) and energy dissipation (right) vs time for modular algorithm with $\gamma = 1$ and $\beta = 0.2$.}\label{figure=edke}
 \end{figure}
\section{Numerical Experiments}
In this section, we illustrate the stability and convergence of the numerical schemes described by (\ref{step1}) - (\ref{step2}) and (\ref{step1}) - (\ref{step2b}).  The numerical experiments include a convergence experiment and timing test using a Taylor-Green benchmark problem.  Next, we consider the flow over a step and flow past a cylinder benchmark problems.  Taylor-Hood (P2-P1) and MINI (P1b-P1) elements are used to approximate the velocity and pressure distributions.  Finite element meshes are Delaunay triangulations generated by $m$ points on each side of the domain.  The software platform used for all tests is \textsc{FreeFem}$++$ \cite{Hecht}.
\subsection{Taylor-Green problem}
In this section, we illustrate convergence rates and speed of the modular algorithm.  We utilize the Taylor-Green vortex problem on the unit square $\Omega = [0,1]^2$ with exact solution
\begin{align*}
u_{1}(x,y,t) = -\cos(\omega \pi x) \sin(\omega \pi y)\exp(-2\omega^{2}\pi^{2}t/\tau),\\
u_{2}(x,y,t) = \sin(\omega \pi x) \cos(\omega \pi y)\exp(-2\omega^{2}\pi^{2}t/\tau),\\
p(x,y,t) = -\frac{1}{4}\big( \cos(2\omega \pi x) + \cos(2\omega \pi y)\big)\exp(-4\omega^{2}\pi^{2}t/\tau).
\end{align*}
We let $\tau = Re = 100$ and $\omega = 1$.  Moreover, $t^{\ast} = 1$ and $\Delta t = 1/m$.  The initial condition is given by the exact solution at $t=0$.  We set $\gamma = 1$ and $\beta = 0.2$ (Algorithm 1), $0$ (Algorithm 2) for convergence rates and vary $m$ between 32, 40, 48, 56, and 64.  Errors are computed for velocity and pressure in the appropriate norms.  The results are presented in Tables \ref{table=p1bp1} and \ref{table=p2p1} for both the MINI and Taylor-Hood elements.  We also compare the kinetic energy and energy dissipation of the true solution, non-stabilized and stabilized linearly implicit BDF1, and Algorithm 1, Figure \ref{figure=edke}.  Interestingly, the kinetic energy and energy dissipation are all in good agreement.  This indicates that, e.g., the extra dissipation introduced from the numerical methods is negligible for the given parameter values.

For the timing test, we fix $\Delta t = 1/m = 1/32$ and vary the grad-div parameters $0 \leq \beta \leq 8,000$ and $0\leq \gamma \leq 20,000$.  GMRES is used for Step 1 and UMFPACK for Step 2.  If GMRES fails at a single iterate, we denote the result with an 'F'.  We also calculate the percent increase in time to compute with the standard implementation relative to the modular implementation.  The results are presented in Table \ref{table=taylorgreen}.
Aside from the normal verification of convergence rates, we note the following.  In the first test with $\gamma = 1$ and $\beta = 0.2$, $\nabla \cdot u^{n}$ is consistently very small.  The second test is of the efficiency gain in the modular vs fully coupled methods.  Both methods can be implemented for increased efficiency and exact timings will be strongly dependent on many factors.  We solve each with standard GMRES.  Since Step 2 in the modular algorithm leads to an SPD system with fixed coefficient matrix, we use a sparse direct method for it.  Again, there are many ways to exploit the structure of Step 2.  Table \ref{table=taylorgreen} is summarized in Figure \ref{figure=speedtest} in the introduction.  Generally, for the fully coupled system, increasing $\gamma$ increases solution time; consistent with what has been reported by others \cite{Bowers,Glowinski}.  GMRES fails to converge within the set maximum iterations at quite moderate values of $\beta$ and $\gamma$.

\begin{table}
	\begin{tabular}{ c  c  c  c  c  c  c  c	c}
		\hline			
		$m$ & $\vertiii{ u_{h}- u }_{\infty,0}$ & Rate & $\vertiii{ \nabla \cdot(u_{h} - u) }_{\infty,0}$ & Rate & $\vertiii{ \nabla \cdot(u_{h} - u) }_{2,0}$ & Rate & $\vertiii{ p_{h} - p }_{2,0}$ & Rate \\
		\hline
		32 & 2.97E-03 & - & 1.13E-01 & - & 9.93E-02 & - & 4.55E-05 & -\\
		40 & 2.22E-03 & 1.31 & 9.08E-02 & 0.98 & 7.98E-02 & 0.98 & 2.00E-05 & 3.67\\
		48 & 1.67E-03 & 1.55 & 7.66E-02 & 0.94 & 6.73E-02 & 0.94 & 1.71E-05 & 0.88\\
		56 & 1.28E-03 & 1.74 & 6.55E-02 & 1.01 & 5.75E-02 & 1.01 & 1.31E-05 & 1.74\\
		64 & 1.02E-03 & 1.69 & 5.73E-02 & 1.01 & 5.03E-02 & 1.01 & 6.85E-06 & 4.82\\
		\hline  
	\end{tabular}
	
	\bigskip
	
	\begin{tabular}{ c  c  c  c  c  c  c  c	c}
			\hline			
			$m$ & $\vertiii{ u_{h}- u }_{\infty,0}$ & Rate & $\vertiii{ \nabla \cdot(u_{h} - u) }_{\infty,0}$ & Rate & $\vertiii{ \nabla \cdot(u_{h} - u) }_{2,0}$ & Rate & $\vertiii{ p_{h} - p }_{2,0}$ & Rate \\
			\hline
			32 & 3.46E-03 & - & 1.62E-01 & - & 1.23E-01 & - & 4.55E-05 & -\\
			40 & 2.77E-03 & 1.00 & 1.29E-01 & 1.02 & 9.87E-02 & 1.00 & 2.00E-05 & 3.67\\
			48 & 2.30E-03 & 1.01 & 1.10E-01 & 0.88 & 8.31E-02 & 0.94 & 1.71E-05 & 0.88\\
			56 & 2.01E-03 & 0.89 & 9.32E-02 & 1.07 & 7.11E-02 & 1.01 & 1.31E-05 & 1.74\\
			64 & 1.68E-03 & 1.33 & 8.10E-02 & 1.05 & 6.20E-02 & 1.02 & 6.85E-06 & 4.83\\
			\hline  
	\end{tabular}
		\caption{Errors and rates for velocity and pressure in corresponding norms of Algorithm 1 (top) and 2 (bottom) for the MINI element.}\label{table=p1bp1}
\end{table}

\begin{table}
	\begin{tabular}{ c  c  c  c  c  c  c  c	c}
		\hline			
		$m$ & $\vertiii{ u_{h}- u }_{\infty,0}$ & Rate & $\vertiii{ \nabla \cdot(u_{h} - u) }_{\infty,0}$ & Rate & $\vertiii{ \nabla \cdot(u_{h} - u) }_{2,0}$ & Rate & $\vertiii{ p_{h} - p }_{2,0}$ & Rate \\
		\hline
		32 & 4.57E-05 & - & 7.49E-04 & - & 6.62E-04 & - & 4.55E-05 & -\\
		40 & 2.97E-05 & 1.93 & 5.04E-04 & 1.78 & 4.44E-04 & 1.79 & 2.00E-05 & 3.67\\
		48 & 2.23E-05 & 1.58 & 3.70E-04 & 1.70 & 3.26E-04 & 1.70 & 1.71E-05 & 0.89\\
		56 & 1.83E-05 & 1.25 & 2.48E-04 & 2.60 & 2.18E-04 & 2.60 & 1.31E-05 & 1.72\\
		64 & 1.54E-05 & 1.32 & 2.01E-04 & 1.56 & 1.77E-04 & 1.57 & 6.86E-06 & 4.83\\
		\hline  
	\end{tabular}
	
	\bigskip
	
	\begin{tabular}{ c  c  c  c  c  c  c  c	c}
		\hline			
		$m$ & $\vertiii{ u_{h}- u }_{\infty,0}$ & Rate & $\vertiii{ \nabla \cdot(u_{h} - u) }_{\infty,0}$ & Rate & $\vertiii{ \nabla \cdot(u_{h} - u) }_{2,0}$ & Rate & $\vertiii{ p_{h} - p }_{2,0}$ & Rate \\
		\hline
		32 & 3.02E-03 & - & 3.19E-03 & - & 2.94E-03 & - & 4.55E-05 & -\\
		40 & 2.47E-03 & 0.91 & 2.30E-03 & 1.46 & 2.12E-03 & 1.47 & 2.00E-05 & 3.67\\
		48 & 2.07E-03 & 0.96 & 1.75E-03 & 1.49 & 1.61E-03 & 1.50 & 1.71E-05 & 0.89\\
		56 & 1.78E-03 & 0.99 & 1.38E-03 & 1.57 & 1.26E-03 & 1.57 & 1.31E-05 & 1.72\\
		64 & 1.56E-03 & 1.01 & 1.12E-03 & 1.54 & 1.03E-03 & 1.54 & 6.86E-06 & 4.83\\
		\hline  
	\end{tabular}
	\caption{Errors and rates for velocity and pressure in corresponding norms of Algorithm 1 (top) and 2 (bottom) for the Taylor-Hood element.}\label{table=p2p1}
\end{table}

\begin{adjustbox}{max width=\textwidth}
	\begin{tabular}{ c  c  c  c  c  c  c  c	c}
		\hline			
		\multicolumn{2}{c} {Parameters} & \multicolumn{3}{c} {P1b-P1} & \multicolumn{3}{c} {P2-P1} \\
		$\beta$ & $\gamma$ & Standard time (s) & Modular time (s) & \% increase & Standard time (s) & Modular time (s) & \% increase \\
		\hline
		0 & 0 & 8.13 & 8.43 & -3.47 & 11.31 & 11.64 & -2.81 \\
		0 & 0.2 & 12.97 & 9.07 & 42.99 & 16.99 & 12.18 & 39.40\\
		0 & 2 & 25.62 & 8.85 & 189.62 & 58.42 & 12.49 & 367.67\\
		0 & 20 & 150.20 & 9.26 & 1522.04 & F & 12.49 & -\\
		0 & 200 & 150.00 & 9.18 & 1533.32 & F & 12.90 & -\\
		0 & 2,000 & 149.84 & 9.24 & 1522.57 & F & 12.71 & -\\
		0 & 20,000 & 149.95 & 9.11 & 1546.71 & F & 12.68 & -\\
		0.01 & 0.2 & 13.58 & 8.48 & 60.19 & 21.00 & 12.11 & 73.51\\
		0.02 & 0.2 & 18.30 & 8.79 & 108.28 & 25.72 & 12.47 & 106.31\\
		0.04 & 0.2 & 14.14 & 8.86 & 59.59 & 47.63 & 12.28 & 287.73\\
		0.08 & 0.2 & F & 8.80 & - & 39.83 & 12.27 & 224.52\\
		0.8 & 0.2 & F & 8.47 & - & F & 12.38 & -\\
		8 & 0.2 & F & 8.79 & - & F & 12.00 & -\\
		80 & 0.2 & F & 8.68 & - & F & 12.35 & -\\
		800 & 0.2 & F & 8.75 & - & F & 12.43 & -\\
		8,000 & 0.2 & F & 8.59 & - & F & 12.61 & -\\
		\hline  
	\end{tabular}
\end{adjustbox}
\captionof{table}{Time to solve the Taylor-Green problem using the MINI (left) and Taylor-Hood (right) elements.  These results are summarized in Figure \ref{figure=speedtest} of the introduction.} \label{table=taylorgreen}
\subsection{2D Channel Flow Over a Step}\setlength\parindent{24pt}
We next study the effects of grad-div stabilization on the quality of the solution and compare fully coupled and modular implementations.  The benchmark problem is 2D channel flow over a step \cite{Fragos,John3}.  A fluid with $\nu = 1/600$ flows within a channel.  The channel dimensions are $40 \times 10$ with a $1 \times 1$ step placed five units into the channel from the l.h.s.  No body forces are imposed, e.g. $f = 0$.  No-slip boundary conditions are imposed at all walls.  Inlet and outlet profiles are prescribed via
\begin{align*}
u(0,y,t)=u(40,y,t)&= y (10 - y) /25, \\
v(0,y,t)=v(40,y,t) &= 0.
\end{align*}
This problem exhibits a smooth velocity distribution with eddy formation and detachment occurring behind the step.
\begin{figure}
	\centering
	\includegraphics[height=3.5in, width=\textwidth]{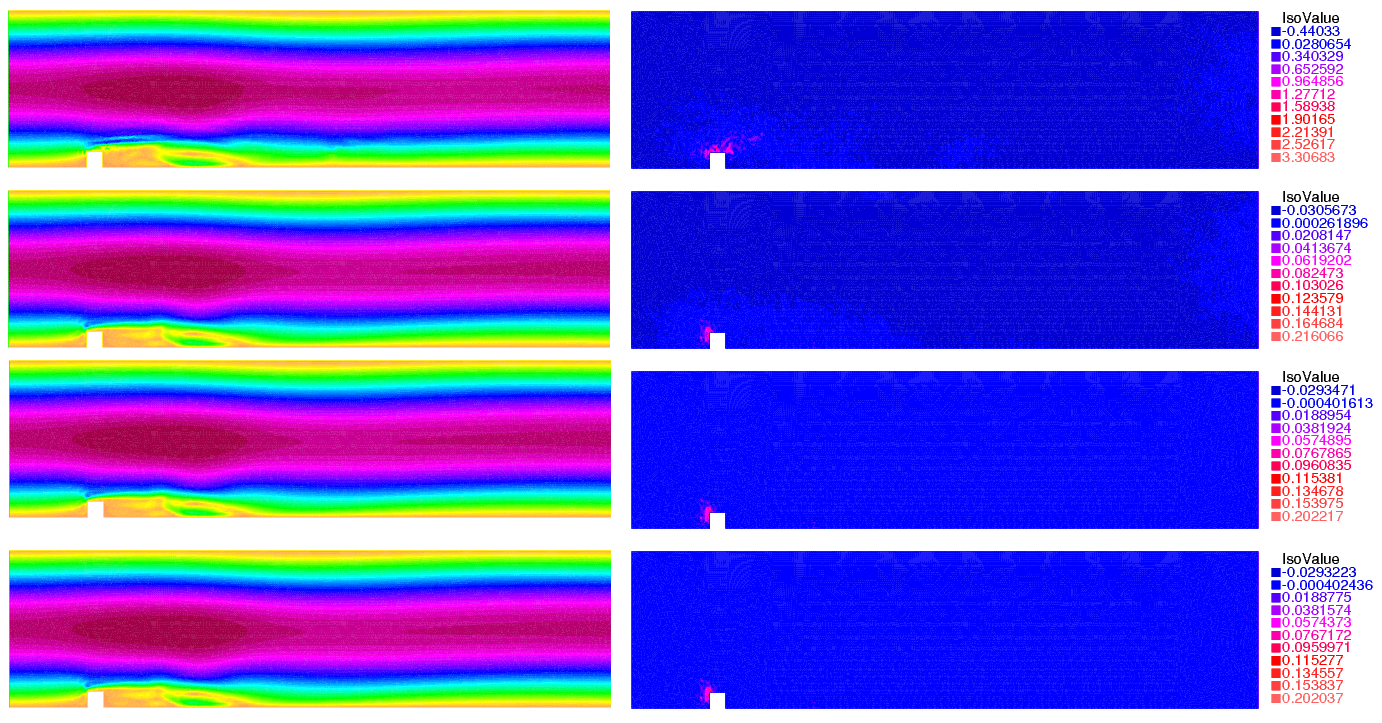}
	\caption{Speed and div contour for flow past a step at times $t^{\ast} = 40$ with $\gamma = 1$ and $\beta = 0$ for non-stabilized, standard, modular, and modular + sparse, respectively.}\label{figure=step}
\end{figure}

\indent Taylor-Hood elements (P2-P1) are used for velocity and pressure on a mesh with 31,233 total degrees of freedom.  The timestep $\Delta t$ is set to 0.01.  We compare numerical approximations using standard implementations of grad-div stabilization with the modular implementations.  The chosen parameters are $\gamma = 0$, $0.1$ \cite{Bowers}, $0.2$ \cite{Linke}, and $1$ \cite{Olshanskii} and $\beta = 0,\; 01,\; 0.2,\; 1$.  When $\gamma = \beta = 0$, this is simply linearly implicit BDF1; we denote this as the non-stabilized solution.  In Figure \ref{figure=step}, we present plots of speed and divergence contours at the final time $t^{\ast} = 40$ with $\beta = 0$ and $\gamma = 1$.  We see that the stabilized solutions are in good agreement with one another.  Moreover, there is a significant reduction ($\simeq 16$) in the divergence error over the non-stabilized solution.  

We also present plots of $\|\nabla \cdot u \|(t)$ vs time for Algorithm 1 in Figure \ref{figure=divcomparison}.
\begin{figure}
	\centering
	\includegraphics[height=2.30in, keepaspectratio]{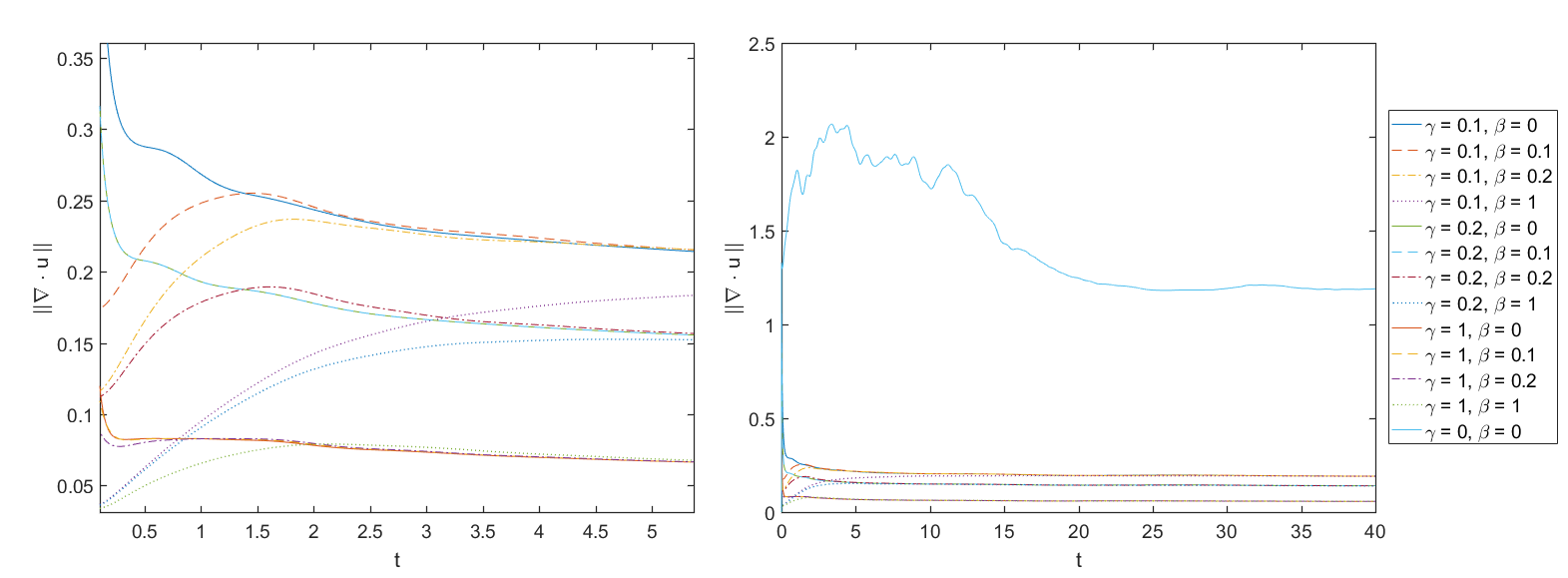}
	\caption{Comparison of $\|\nabla \cdot u\|(t)$ vs time for the modular implementation (right) and zoomed in (left).}\label{figure=divcomparison}
\end{figure}
Comparing the non-stabilized solution with stabilized solutions, we see an improvement in the divergence error.  In particular, when $\gamma = 1$ the divergence error is significantly reduced.  Moreover, when $\beta = 1$, the divergence error is further reduced.  Interestingly, we see that the value of $\beta$ has a strong effect early in the simulation while $\gamma$ determines the long-term behavior.
\subsection{2D Channel Flow Past a Cylinder}\setlength\parindent{24pt}
As a third validation experiment, we compare standard (full coupling) and modular implementations for channel flow past a cylinder \cite{Schafer}.  Once again, we find that modular implementations produce results consistent with full coupling.  A fluid with $\nu = 0.001$ and $\rho = 1$ flows within a channel.  The channel has dimensions $2.2 \times 0.41$.  A cylinder of diameter $0.1$ is centered at (0.2,0.2) within the channel.  No body forces are assumed present.  No-slip boundary conditions are imposed at all walls.  Inlet and outlet profiles are prescribed via
\begin{align*}
u(0,y,t)=u(2.2,y,t) &=\frac{6y(0.41-y)}{0.41^2} \sin(\pi t/8),\\
v(0,y,t)=v(2.2,y,t) &= 0.
\end{align*}
The solution is interesting, involving the development of two vortices and eventually a vortex street which persists throughout the simulation time.

\indent Taylor-Hood elements (P2-P1) are used on a mesh with 64,554 total degrees of freedom.  The timestep is $\Delta t = 0.001$.  The chosen parameters are $\gamma = 5\nu$ and $\beta = 0$.  Drag $c_{d}(t)$ and lift $c_{l}(t)$ coefficients are calculated as well as the pressure difference between the front and back of the cylinder $\Delta p (t) = p (0.15,0.2,t)- p(0.25,0.2,t)$.  We compare the computed values with the accepted values seen in Schafer and Turek \cite{Schafer}; in particular, $c^{max}_{d} \in [2.93,2.97]$, $c^{max}_{l} \in [0.47,0.49]$, and $\Delta p (8) \in [-0.115,-0.105]$.  We also compute $\| \nabla \cdot u \|(t)$.  These are presented in Table \ref{table=cylinder}.  

\indent Lastly, we present plots of velocity and speed contours for the modular algorithm at times $t^{\ast} = 4,\; 6,\; 7,\; 8$ in Figure \ref{figure=cylinder}; these are consistent with what is seen in the literature \cite{Bowers,John4,Linke}.  In Table \ref{table=cylinder}, we see that all algorithms produce consistent values.  We note that pressure drop across the cylinder for the non-stabilized and sparse algorithms are slightly off.  Moreover, the non-stabilized solution consistently produces larger divergence error, illustrated by the last column.
\vspace{5mm}

\begin{figure}
	\centering
	\includegraphics[height=7.0in, keepaspectratio]{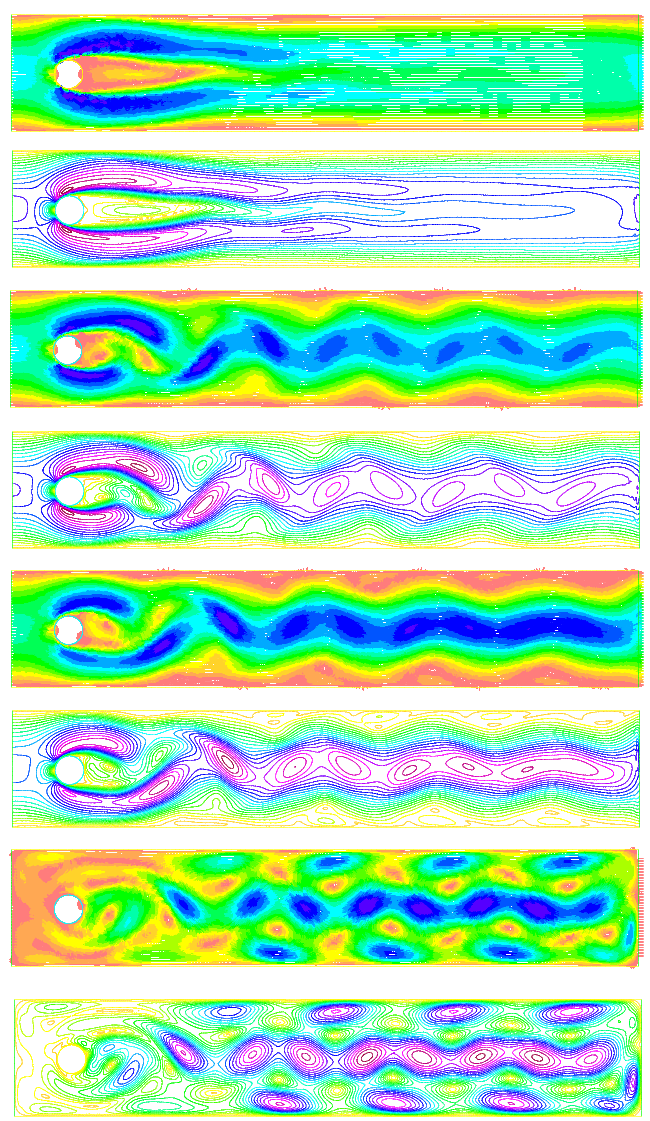}
	\caption{Velocity and speed contour for flow past a cylinder at times $t^{\ast}$ = 4, 6, 7, and 8.}\label{figure=cylinder}
\end{figure}
{
\centering
\begin{adjustbox}{max width=\textwidth}
	\begin{tabular}{ c	c	c	c	c}
		\hline			
		Method & $c^{max}_{d}$ & $c^{max}_{l}$ & $\Delta p (8)$ & $\| \nabla \cdot u \|(8)$\\
		\hline
		non-Stabilized & 2.950 & 0.471 & \textbf{-0.1046} & \textbf{0.083} \\
		Standard & 2.950 & 0.478 & -0.1056 & 0.030 \\
		Modular & 2.950 & 0.478 & -0.1055 & 0.029 \\
		Modular + Sparse & 2.950 & 0.486 & \textbf{-0.1037} & 0.029 \\
		\hline  
	\end{tabular}
\end{adjustbox}
\captionof{table}{Maximum lift, drag coefficients and pressure drop across cylinder for flow past a cylinder.}\label{table=cylinder}
}

\section{Conclusion}
We presented two modular implementations of grad-div stabilization for fluid flow problems.  The presented algorithms presume a standard NSE code and add a minimally intrusive module that has the effect of adding dissipative or dispersive grad-div stabilizations.  Both modular algorithms are proven to be stable and optimally convergent. Moreover, numerical experiments were performed to illustrate the above properties and the value of the modular implementations.  In particular, it is shown that in solution quality and imposition of mass conservation, the modular implementations perform as well as standard implementations.  Further, in our tests, the modular methods are unaffected by variations of grad-div stabilization parameters whereas the cost of standard implementations grow rapidly as the parameters grow.


\end{document}